\documentclass[12pt]{amsart}

\usepackage{url}
\usepackage{fullpage}
\usepackage{amssymb}
\usepackage{color}

\usepackage{tikz}
\usetikzlibrary{decorations.markings}

\theoremstyle{plain}
\newtheorem{theorem}{Theorem}[section]

\newtheorem{corollary}[theorem]{Corollary}

\newtheorem{lemma}[theorem]{Lemma}

\newtheorem{proposition}[theorem]{Proposition}

\newtheorem{problem}[theorem]{Problem}

\newtheorem{construction}[theorem]{Construction}

\theoremstyle{definition}

\newtheorem{algorithm}[theorem]{Algorithm}

\numberwithin{equation}{section}

\def\ldiv{\backslash}
\def\rdiv{/}
\def\genof#1#2{\langle #1\,:\,#2\rangle}
\def\setof#1#2{\{ #1\,:\,#2 \}}
\def\aut#1{\mathrm{Aut}(#1)}
\def\mlt#1{\mathrm{Mlt}(#1)}
\def\inn#1{\mathrm{Inn}(#1)}
\def\fix#1{\mathrm{Fix}(#1)}
\def\rad#1{\mathrm{Rad}(#1)}
\def\lmlt#1{\mathrm{Mlt}_\ell(#1)}
\def\linn#1{\mathrm{Inn}_\ell(#1)}
\def\sym#1{\mathrm{Sym}(#1)}
\def\lad{\mathrm{ad}^{\ell}}
\def\rad{\mathrm{ad}^{r}}
\def\id{\mathrm{id}}

\begin{document}

\title{Three lectures on automorphic loops}

\author{Petr Vojt\v{e}chovsk\'y}

\email{petr@math.du.edu}

\address{Department of Mathematics, University of Denver, 2280 S Vine St, Denver, Colorado 80112, U.S.A.}

\begin{abstract}
These notes accompany a series of three lectures on automorphic loops to be delivered by the author at Workshops Loops '15 (Ohrid, Macedonia, 2015). Automorphic loops are loops in which all inner mappings are automorphisms.

The first paper on automorphic loops appeared in 1956 and there has been a surge of interest in the topic since 2010. The purpose of these notes is to introduce the methods used in the study of automorphic loops to a wider audience of researchers working in nonassociative mathematics.

In the first lecture we establish basic properties of automorphic loops (flexibility, power-associativity and the antiautomorphic inverse property) and discuss relations of automorphic loops to Moufang loops.

In the second lecture we expand on ideas of Glauberman and investigate the associated operation $(x^{-1}\ldiv (y^2x))^{1/2}$ and similar concepts, using a more modern approach of twisted subgroups. We establish many structural results for commutative and general automorphic loops, including the Odd Order Theorem.

In the last lecture we look at enumeration and constructions of automorphic loops. We show that there are no nonassociative simple automorphic loops of order less than $4096$, we study commutative automorphic loops of order $pq$ and $p^3$, and introduce two general constructions of automorphic loops.

The material is newly organized and sometimes new, shorter proofs are given.
\end{abstract}

\thanks{Research partially supported by the Simons Foundation Collaboration Grant 210176.}



\maketitle

\tableofcontents

\section*{Introduction}

The purpose of these notes is to give a gentle introduction into the theory of automorphic loops that nevertheless captures the main ideas of current investigation. Due to the limited scope of the lectures, not all proofs are included and not all known results about automorphic loops are stated. A survey article on automorphic loops that attempts to remedy both of these shortcomings is under preparation by the author and will appear elsewhere.

\medskip

Let $Q = (Q,\cdot,\ldiv,\rdiv,1)$ be a loop, where we also write $xy$ to denote the product $x\cdot y$. For $x\in Q$, let
\begin{displaymath}
    L_x:Q\to Q,\ L_x(y) = xy\qquad\text{and}\qquad R_x:Q\to Q,\ R_x(y) = yx
\end{displaymath}
be the \emph{left} and \emph{right translation by $x$}, respectively. The permutation group
\begin{displaymath}
    \mlt{Q} = \genof{L_x,\,R_x}{x\in Q}
\end{displaymath}
is called the \emph{multiplication group of $Q$}, and its subloop
\begin{displaymath}
    \inn{Q} = \genof{\varphi\in\mlt{Q}}{\varphi(1)=1}
\end{displaymath}
is the \emph{inner mapping group of $Q$}.

Denote by $\aut{Q}$ the automorphism group of $Q$. An \emph{automorphic loop} (or \emph{A-loop}) is a loop $Q$ in which every inner mapping is an automorphism, that is, $\inn{Q}\le\aut{Q}$. Note that groups are automorphic loops, but the converse is certainly not true.

The following multiplication table specifies a nonassociative automorphic loop of the smallest possible order, which we will call $Q_6$:
\begin{displaymath}
\begin{array}{c|cccccc}
    Q_6&1&2&3&4&5&6\\
    \hline
    1&1&2&3&4&5&6\\
    2&2&1&4&6&3&5\\
    3&3&5&1&2&6&4\\
    4&4&3&6&5&1&2\\
    5&5&6&2&1&4&3\\
    6&6&4&5&3&2&1
\end{array}\ .
\end{displaymath}
Properties of $Q_6$ can be checked in the \texttt{GAP} \cite{GAP} package \texttt{LOOPS} \cite{LOOPS}, for instance.

Bruck proved \cite{Bruck} that in any loop
\begin{displaymath}
    \inn{Q} = \genof{L_{x,y},\,R_{x,y},\,T_x}{x,\,y\in Q},
\end{displaymath}
where
\begin{displaymath}
    L_{x,y}(z) = (yx)\ldiv (y(xz)),\quad R_{x,y}(z) = ((zx)y)\rdiv(xy),\quad\text{and}\quad T_x(y) = x\ldiv(yx).
\end{displaymath}
It is also well known that a mapping between two loops is a homomorphism of loops if and only if it respects the multiplication operation. Because this fact is of crucial importance for automorphic loops, we give a short proof:

Let $f:(A,\cdot_A,\ldiv_A,\rdiv_A,1_A)\to (B,\cdot_B,\ldiv_B,\rdiv_B,1_B)$ be a mapping between loops such that $f(x\cdot_A y) = f(x)\cdot_B f(y)$ for every $x$, $y\in A$. Then $f(x)\cdot_B f(x\ldiv_A y) = f(x\cdot_A (x\ldiv_A y)) = f(y)$ and therefore $f(x\ldiv_A y) = f(x)\ldiv_B f(y)$ for every $x$, $y\in Q$. The argument for right division is dual, and the property $f(1_A) = 1_B$ is obtained by cancelation from $f(1_A) = f(1_A\cdot_A 1_A) = f(1_A)\cdot_B f(1_A)$.

It follows that a loop $Q$ is an automorphic loop if and only if for every $x$, $y\in Q$ the inner mappings $L_{x,y}$, $R_{x,y}$ and $T_x$ respect multiplication. Consequently, the class of automorphic loops is a subvariety of the variety of loops, consisting of all loops satisfying the axioms
\begin{align}
    (yx)\ldiv (y(x(uv))) &= ((yx)\ldiv (y(xu)))((yx)\ldiv (y(xv))),\tag{A$_\ell$}\label{Eq:Al}\\
    (((uv)x)y)\rdiv(xy) &= (((ux)y)\rdiv(xy))(((vx)y)\rdiv(xy)),\tag{A$_r$}\label{Eq:Ar}\\
    x\ldiv((uv)x) &= (x\ldiv(ux))(x\ldiv(vx)).\tag{A$_m$}\label{Eq:Am}
\end{align}
In particular, subloops, factor loops and homomorphic images of automorphic loops are again automorphic loops.

We call a loop \emph{left automorphic} if \eqref{Eq:Al} holds, \emph{right automorphic} if \eqref{Eq:Ar} holds, and \emph{middle automorphic} if \eqref{Eq:Am} holds.

\medskip

The axioms \eqref{Eq:Al}, \eqref{Eq:Ar}, \eqref{Eq:Am} are somewhat long and intimidating, certainly much more so than the axiom
\begin{equation}\label{Eq:Moufang}
    (xy)(zx) = (x(yz))x\tag{M}
\end{equation}
defining Moufang loops, for instance. But the message of the axioms is easy to remember---``inner mappings respect multiplication''---and, as we shall see, automorphic loops are very much amenable to algebraic investigation.

\medskip

Such an investigation started in earnest in 1956 with the work of Bruck and Paige \cite{BrPa}. We will retrace some of their steps, for instance by proving that automorphic loops are power-associative. The main contribution of \cite{BrPa}, which we will not follow here, was to demonstrate that diassociative automorphic loops share many properties with Moufang loops (which are always diassociative, by Moufang's theorem \cite{Moufang}).

The conjecture that every diassociative automorphic loop is Moufang is implicit in \cite{BrPa}, but its proof remained elusive for 45 years. The conjecture was established for the special case of commutative loops by Osborn in 1958 \cite{Osborn}. Since commutative Moufang loops are automorphic by \cite{Bruck} (or see Proposition \ref{Pr:CML}), it follows from Osborn's result that commutative Moufang loops are precisely commutative diassociative automorphic loops. The full conjecture was finally confirmed by Kinyon, Kunen and Phillips in 2002 \cite{KiKuPh}.

Following a few sporadic results in late 1900s and early 2000s, of which we mention \cite{DrapalAloops, Figula, Kikkawa, NaSt, Shchukin}, automorphic loops became one of the focal areas in loop theory after the work of Jedli\v{c}ka, Kinyon and the author on commutative automorphic loops \cite{JeKiVoStructure, JeKiVoConstructions} was circulated. It is worth mentioning that some results of \cite{JeKiVoStructure} were first obtained by automated deduction \cite{McCune}, which remains influential in this field. But once the initial hurdles were cleared, the theory opened up to more traditional modes of investigation.

New results by many authors followed in quick succession. We mention two highlights: Odd Order Theorem for automorphic loops \cite{KiKuPhVo}, and solvability of finite commutative automorphic loops \cite{GrKiNa}.

The field remains active and we hope that these survey notes will attract new researchers to automorphic loops and related areas. Open problems can be found in the last section of this paper.

\medskip

From now on we will employ the following notational conventions in order to save parentheses and improve legibility. The division operations are less binding than juxtaposition, and the explicit $\cdot$ multiplication is less binding than divisions and juxtaposition. For instance, $x\rdiv y \cdot y\ldiv zy$ means $(x\rdiv y)(y\ldiv(zy))$.

\section*{Lecture 1: Basic properties}
\setcounter{section}{1}
\setcounter{theorem}{0}

In this section we establish some basic properties of automorphic loops. Most of these properties were known already to Bruck and Paige \cite{BrPa}, except that they were not aware of the fact that automorphic loops have the antiautomorphic inverse property (see \cite{JoKiNaVo} or Proposition \ref{Pr:AAIP}) and its consequences (one of the axioms \eqref{Eq:Al}, \eqref{Eq:Ar} can be ommitted by Theorem \ref{Th:LMEnough}, and the left and right nuclei coincide by Theorem \ref{Th:Nuclei}). Of course, they also did not know that diassociative automorphic loops are Moufang \cite{KiKuPh}, a result that we have incorporated without proof into Theorem \ref{Th:DAMoufang}.

Many proofs presented in this section shorten older arguments. We do not hesitate to prove even folklore results to better show to the reader that most result in this section can be derived quickly from first principles. In this spirit, consider:

\begin{lemma}\label{Lm:Action}
Let $Q$ be a loop and $\varphi\in\aut{Q}$. Then
\begin{gather*}
    \varphi L_x^{\pm 1}\varphi^{-1} = L_{\varphi(x)}^{\pm 1},\quad \varphi R_x^{\pm 1}\varphi^{-1} = R_{\varphi(x)}^{\pm 1},\\
    \varphi T_x^{\pm 1}\varphi^{-1} = T_{\varphi(x)}^{\pm 1},\quad \varphi L_{x,y}^{\pm 1}\varphi^{-1} = L_{\varphi(x),\varphi(y)}^{\pm 1},\quad \varphi R_{x,y}^{\pm 1}\varphi^{-1} = R_{\varphi(x),\varphi(y)}^{\pm 1}
\end{gather*}
for every $x$, $y\in Q$.
\end{lemma}
\begin{proof}
We have $\varphi L_x\varphi^{-1}(y) = \varphi( x\cdot \varphi^{-1}(y) ) = \varphi(x)\cdot \varphi(\varphi^{-1}(y)) = \varphi(x)\cdot y = L_{\varphi(x)}(y)$, so $\varphi L_x\varphi^{-1} = L_{\varphi(x)}$. Then $\varphi L_x^{-1}\varphi^{-1} = (\varphi L_x\varphi^{-1})^{-1} = L_{\varphi(x)}^{-1}$. The argument for $R_x$ is similar. Then $\varphi T_x\varphi^{-1} = \varphi L_x^{-1}R_x\varphi^{-1} = \varphi L_x^{-1}\varphi^{-1}\varphi R_x\varphi^{-1} = L_{\varphi(x)}^{-1} R_{\varphi(x)} = T_{\varphi(x)}$, and so on.
\end{proof}

Thus in any loop $Q$, the automorphism group $\aut{Q}$ acts on $\mlt{Q}$ and on $\inn{Q}$ by conjugation, mapping left inner mappings to left inner mappings, and so on. If $Q$ is an automorphic loop, then the action of $\aut{Q}$ induces an action of $\inn{Q}$.

\subsection{Flexibility and power-associativity}

A loop $Q$ is \emph{flexible} if $x(yx)=(xy)x$ holds for every $x$, $y\in Q$. A consequence of flexibility is that every element $x$ has a (unique) two-sided inverse $x^{-1}$. Indeed, if $x^\ell$, $x^r$ ar the left and right inverses of $x$, respectively, then $x = x(x^\ell x) = (xx^\ell)x$, so $xx^\ell = 1 = xx^r$ and $x^\ell=x^r$.

\begin{proposition}[{{\cite[p.\@ 311]{BrPa}}}]\label{Pr:Flexible}
Every middle automorphic loop is flexible.
\end{proposition}
\begin{proof}
Suppose that $Q$ satisfies \eqref{Eq:Am}. Then $T_x(xy) = T_x(x)\cdot T_x(y)$, and multiplying this equality by $x$ on the left yields $(xy)x = x(x\ldiv xx \cdot x\ldiv yx) = x(x\cdot x\ldiv yx) = x(yx)$.
\end{proof}

We remark that there exists a loop (of order $6$) that is both left and right automorphic, yet does not posses two-sided inverses, so is also not flexible.

A loop $Q$ is said to be \emph{power-associative} if for every $x\in Q$ the subloop $\langle x\rangle$ of $Q$ generated by $x$ is associative. For a prime $p$, a power-associative loop $Q$ is said to be a \emph{$p$-loop} if every element of $Q$ has order that is a power of $p$.

Assuming two-sided inverses, a general strategy for proving power-associativity is as follows: Define \emph{nominal powers} $x^{[n]}$ by letting $x^{[0]}=1$, $x^{[k+1]}=xx^{[k]}$ and $x^{[-k]} = (x^{[k]})^{-1}$. Then it is not hard to show by induction that $Q$ is power-associative if and only if
\begin{equation}\label{Eq:PA}
    x^{[i]}(x^{[j]}x^{[k]}) = (x^{[i]}x^{[j]})x^{[k]}
\end{equation}
for all $i$, $j$, $k\in\mathbb Z$. A typical proof of \eqref{Eq:PA} in a given variety of loops is based on a careful induction. In automorphic loops, however, Bruck and Paige \cite{BrPa} employed an ingenious argument that we will essentially follow here.

Note that for any loop $Q$ and a subset $A$ of $\aut{Q}$ the set
\begin{displaymath}
    \fix{A} = \setof{x\in Q}{\varphi(x)=x\text{ for every $\varphi\in A$}}
\end{displaymath}
of common fixed points of automorphisms from $A$ is a subloop of $Q$.

\begin{proposition}[{{\cite[Theorems 2.4 and 2.6]{BrPa}}}]\label{Pr:PowerAssociative}
Every automorphic loop is power-associative and satisfies $(x^iy)x^j = x^i(yx^j)$, $x^i(x^jy) = x^j(x^iy)$, $(yx^i)x^j = (yx^j)x^i$ for every $i$, $j\in\mathbb Z$.
\end{proposition}
\begin{proof}
Our loop $Q$ is flexible by Proposition \ref{Pr:Flexible}, which implies that $x\in\fix{L_{y,x}}$ and hence $\langle x\rangle\le\fix{L_{y,x}}$. In particular, $(xy)x^{[j]} = x(yx^{[j]})$. (Note that we have not used \eqref{Eq:Ar} yet.) This means that the inner mapping $R_{yx^{[j]}}^{-1}R_{x^{[j]}}R_y$ fixes $x$, thus also $x^{[i]}$, and we have $(x^{[i]}y)x^{[j]} = x^{[i]}(yx^{[j]})$. As a special case we obtain \eqref{Eq:PA}, which implies power-associativity. Then $x^i$ is well-defined, coincides with $x^{[i]}$, and $(x^iy)x^j = x^i(yx^j)$ follows.

The inner mapping $R_{xy}^{-1}L_xR_y$ trivially fixes $x$, so also $x^i$. This shows that $R_{x^iy}^{-1}L_{x^i}R_{y}$ fixes $x$, so also $x^j$, and $x^i(x^jy) = x^j(x^iy)$ follows. The last identity is proved dually.
\end{proof}

Note that the identities of Proposition \ref{Pr:PowerAssociative} say that for a fixed $x$ in an automorphic loop $Q$, the group $\genof{L_{x^i},\,R_{x^i}}{i\in\mathbb Z}$ is commutative.

\subsection{Anti-automorphic inverse property}

A loop with two-sided inverses has the \emph{antiautomorphic inverse property} if it satisfies the identity
\begin{equation}\label{Eq:AAIP}
    (xy)^{-1} = y^{-1}x^{-1}.
\end{equation}
We are now going to show that every automorphic loop has the antiautomorphic inverse property. For reasons that become clear, we prove a seemingly stronger result, assuming only \eqref{Eq:Al} and \eqref{Eq:Am}. We give a shorter proof than in \cite{JoKiNaVo}.

\begin{proposition}[compare {{\cite[Proposition 7.4]{JoKiNaVo}}}]\label{Pr:AAIP}
Every loop that is both left and middle automorphic has the antiautomorphic inverse property.
\end{proposition}
\begin{proof}
In the proof of Proposition \ref{Pr:PowerAssociative} we established $(xy)x^{[j]} = x(yx^{[j]})$ using only \eqref{Eq:Al} and \eqref{Eq:Am}. In particular, we can use $(xy)x^{-1} = x(yx^{-1})$ below. Consider $\psi = L_y^{-1}L_x L_{x\ldiv y} = L_{x\ldiv y,x}\in\aut{Q}$. Since $\psi((x\ldiv y)^{-1}) = y\ldiv x$, we also have $\psi(x\ldiv y) = (y\ldiv x)^{-1}$. Then $(y\ldiv x)^{-1}\cdot y^{-1} = (y\ldiv x)^{-1}\cdot y\ldiv 1 = \psi(x\ldiv y)\psi((x\ldiv y)\ldiv x^{-1}) = \psi(x^{-1}) = y\ldiv (x\cdot (x\ldiv y)x^{-1}) = y\ldiv (x(x\ldiv y)\cdot x^{-1}) = y\ldiv yx^{-1} = x^{-1}$. Then \eqref{Eq:AAIP} follows by substituting $yx$ for $x$.
\end{proof}

In general, the antiautomorphic inverse property has a similar effect as commutativity in the sense that it allows one to deduce properties about right concepts from properties of left concepts, and vice versa. In the following well-known lemma, let $J$ be the inversion mapping $x\mapsto x^{-1}$.

\begin{lemma}\label{Lm:J}
Let $Q$ be an antiautomorphic inverse property loop. Then the inversion mapping $J$ is an involutory antiautomorphism of $Q$. Moreover, $JL_xJ = R_{x^{-1}}$ and $JL_{x,y}J = R_{x^{-1},y^{-1}}$ for every $x$, $y\in Q$.
\end{lemma}
\begin{proof}
With $x$, $y\in Q$ we have $JL_xJ(y) = (xy^{-1})^{-1} = yx^{-1} = R_{x^{-1}}(y)$, so $JL_xJ = R_{x^{-1}}$. Then $JL_{x,y}J = JL_{yx}^{-1}J\cdot JL_yJ\cdot JL_xJ = R_{(yx)^{-1}}^{-1} R_{y^{-1}} R_{x^{-1}} = R_{x^{-1}y^{-1}}^{-1} R_{y^{-1}}R_{x^{-1}} = R_{x^{-1},y^{-1}}$.
\end{proof}

We now easily arrive at the following important result:

\begin{theorem}[compare {{\cite[Theorem 7.1]{JoKiNaVo}}}]\label{Th:LMEnough}
The following properties are equivalent for a loop $Q$:
\begin{enumerate}
\item[(i)] $Q$ is automorphic,
\item[(ii)] $Q$ is left and middle automorphic,
\item[(iii)] $Q$ is right and middle automorphic.
\end{enumerate}
\end{theorem}
\begin{proof}
Thanks to the duality, it suffices to establish the implication (ii) $\Rightarrow$ (i). Suppose that $Q$ is left and middle automorphic. By Proposition \ref{Pr:AAIP}, $Q$ has the antiautomorphic inverse property. By Lemma \ref{Lm:J}, $J$ is an antiautomorphism and $R_{x^{-1},y^{-1}} = JL_{x,y}J$ is an automorphism, being a composition of an automorphism and two antiautomorphisms.
\end{proof}

We can further exploit the inversion mapping $J$.

\begin{lemma}[{{\cite[Lemma 2.7]{KiKuPhVo}}}]\label{Lm:LRT}
Let $Q$ be an automorphic loop. Then $J$ centralizes $\inn{Q}$. Moreover, $L_{x,y}=R_{x^{-1},y^{-1}}$ and $T_x^{-1} = T_{x^{-1}}$ for every $x$, $y\in Q$.
\end{lemma}
\begin{proof}
Because $\varphi(x^{-1}) = \varphi(x)^{-1}$ for every $x\in Q$ and $\varphi\in\aut{Q}$, the inversion mapping $J$ centralizes $\inn{Q}\le \aut{Q}$. Combining this with Lemma \ref{Lm:J} yields $L_{x,y} = JL_{x,y}J = R_{x^{-1},y^{-1}}$. Using this fact and Proposition \ref{Pr:PowerAssociative} yields $T_xT_{x^{-1}} = L_x^{-1}R_xL_{x^{-1}}^{-1}R_{x^{-1}} = R_xR_{x^{-1}}L_x^{-1}L_{x^{-1}}^{-1} = R_{x^{-1},x} L_{x,x^{-1}}^{-1} = R_{x^{-1},x}R_{x^{-1},x}^{-1} = 1$.
\end{proof}

\subsection{Nuclei}

As usual, define the \emph{left}, \emph{middle} and \emph{right nucleus} of a loop $Q$ by
\begin{align*}
    N_\ell(Q) &= \setof{a\in Q}{a(xy)=(ax)y\text{ for all }x,\,y\in Q},\\
    N_m(Q) &= \setof{a\in Q}{x(ay)=(xa)y\text{ for all }x,\,y\in Q},\\
    N_r(Q) &= \setof{a\in Q}{x(ya)=(xy)a\text{ for all }x,\,y\in Q},
\end{align*}
respectively, and the \emph{nucleus} of $Q$ by $N(Q) = N_\ell(Q)\cap N_m(Q)\cap N_r(Q)$.

It is easy to observe that all the nuclei are associative subloops of $Q$. In general loops, there is no relationship between the three nuclei $N_\ell(Q)$, $N_m(Q)$ and $N_r(Q)$. On the other hand, it is well known (see below) that in inverse property loops all nuclei coincide.

Recall that a loop with two-sided inverses has the \emph{left inverse property} if $x^{-1}(xy) = y$ holds, the \emph{right inverse property} if $(xy)y^{-1}=x$ holds, and the \emph{inverse property} if it has both the left and right inverse properties.

\begin{proposition}[{{\cite[Theorem VII.2.1]{Bruck}}}]\label{Pr:IPNuclei}
In antiautomorphic inverse property loops the left and right nuclei coincide. In inverse property loops all nuclei coincide.
\end{proposition}
\begin{proof}
Suppose that $Q$ satisfies \eqref{Eq:AAIP}. Then the condition $ax\cdot y = a\cdot xy$ is equivalent to $y^{-1}\cdot x^{-1}a^{-1} = y^{-1}x^{-1}\cdot a^{-1}$, so $N_\ell(Q)=N_r(Q)$. Now suppose that $Q$ has the inverse property. From $(xy)^{-1}x = (xy)^{-1}(xy\cdot y^{-1}) = y^{-1}$ we deduce \eqref{Eq:AAIP}, so it remains to show that $N_\ell(Q)=N_m(Q)$. If $ax\cdot y = a\cdot xy$ then $y = (ax)^{-1}(a\cdot xy)$, and substituting $x=a^{-1}u^{-1}$, $y=ua\cdot v$ yields $ua\cdot v = y = u\cdot av$. The other inclusion follows by a similar argument.
\end{proof}

Suppose that $Q$ is an automorphic loop. We know from Proposition \ref{Pr:AAIP} that $Q$ has the antiautomorphic inverse property, and thus that $N_\ell(Q) = N_r(Q)$ by Proposition \ref{Pr:IPNuclei}. But taking $x=2$ and $y=3$ in $Q_6$ shows that $Q$ does not necessarily have the left or right inverse property, so there is no \emph{a priori} reason why the nuclei of $Q$ should coincide. In fact, there are automorphic loops $Q$ satisfying the strict inclusion $N(Q) = N_\ell(Q)=N_r(Q) < N_m(Q)$. Theorem \ref{Th:Nuclei} shows that no other inclusions among nuclei arise in automorphic loops.

Call a subloop $S$ of a loop $Q$ \emph{characteristic} if $\varphi(S) = S$ for every $\varphi\in\aut{Q}$.

In general loops, nuclei are not necessarily normal subloops, but they are always characteristic subloops. For instance, if $a\in N_\ell(Q)$ and $\varphi\in\aut{Q}$ then $\varphi(a)\cdot \varphi(x)\varphi(y) = \varphi(a\cdot xy) = \varphi(ax\cdot y) = \varphi(a)\varphi(x)\cdot \varphi(y)$ shows that $\varphi(a)\in N_\ell(Q)$.

In automorphic loops, nuclei are therefore normal subloops thanks to this easy but important fact:

\begin{lemma}[{{\cite[Theorem 2.2]{BrPa}}}]\label{Lm:CharNormal}
Let $Q$ be an automorphic loop and $S$ a characteristic subloop of $Q$. Then $S$ is normal in $Q$.
\end{lemma}
\begin{proof}
A subloop $S$ is normal in $Q$ if and only if $\varphi(S)=S$ for every $\varphi\in\inn{Q}$.
\end{proof}

\begin{lemma}\label{Lm:Aux}
Let $Q$ be an automorphic loop. Then $T_xT_y(a) = T_{yx}(a)$ for every $a\in N_\ell(Q)=N_r(Q)$.
\end{lemma}
\begin{proof}
We have already shown that $N_\ell(Q)=N_r(Q)$ is a characteristic subloop of $Q$. Let $u=T_x(y)$ (that is, $xu=yx$). Because $a\in N_r(Q)$, we also have $T_{xu}(a)\in N_r(Q)$, and so $x(uT_{xu}(a)) = (xu)T_{xu}(a) = a(xu)$. Since $a\in N_\ell(Q)$, we then have $T_xT_y(a) = T_x(y\ldiv ay) = T_x(y)\ldiv T_x(ay) = T_x(y)\ldiv (x\ldiv (ay)x) = T_x(y)\ldiv (x\ldiv a(yx)) = u\ldiv (x\ldiv a(xu)) = u\ldiv (x\ldiv x(uT_{xu}(a))) = T_{xu}(a) = T_{yx}(a)$.
\end{proof}

\begin{theorem}\label{Th:Nuclei}
Let $Q$ be an automorphic loop. Then $N(Q) = N_\ell(Q) = N_r(Q) \le N_m(Q)$ and all nuclei are normal subloops of $Q$.
\end{theorem}
\begin{proof}
All nuclei are normal by Lemma \ref{Lm:CharNormal}. Let $A=N_\ell(Q) = N_r(Q)$. It remains to prove that $A\le N_m(Q)$. Note that $L_{x,y}$ and $R_{x,y}$ fix $A$ pointwise, while $(xa)y=x(ay)$ holds if and only if $M_{x,y}(a) = a$, where $M_{x,y} = L_x^{-1}R_y^{-1}L_xR_y$.

Given $a\in A$, we want to show that $M_{x,y}(a)=a$. Now,
\begin{displaymath}
    M_{x,y} = (L_x^{-1}R_x)(R_x^{-1}R_y^{-1}R_{xy})(R_{xy}^{-1}L_{xy})(L_{xy}^{-1}L_xL_y)(L_y^{-1}R_y),
\end{displaymath}
and thus $M_{x,y} = T_x R_{x,y}^{-1}T_{xy}^{-1}L_{y,x}T_y$. While evaluating $M_{x,y}$ at $a$, we never leave the normal subloop $A$, so $M_{x,y}(a) = T_xT_{xy}^{-1}T_y(a)$. By Lemma \ref{Lm:Aux}, $M_{x,y}(a)=T_xT_{xy}^{-1}T_y(a) = T_x(T_yT_x)^{-1}T_y(a) = a$.
\end{proof}

The middle nucleus is important in automorphic loops but its role is not fully understood.

\subsection{Diassociativity and the Moufang property}

Up to this point we have carefully proved all the results. In this subsection we skip some proofs and refer the reader to the literature.

A loop has the \emph{left alternative property} if it satisfies $x(xy)=(xx)y$ and the \emph{right alternative property} if $x(yy) =(xy)y$ holds. A loop $Q$ is \emph{diassociative} if any two elements of $Q$ generate an associative subloop.

By Moufang's theorem \cite{Moufang}, Moufang loops are diassociative. The loop $Q_6$ with $x=2$ and $y=3$ shows that automorphic loops need not have the left alternative property nor the right alternative property so, in particular, they need not be diassociative.

Bruck and Paige proved in \cite[Theorem 2.4]{BrPa} that the following properties are equivalent for an automorphic loop $Q$: $Q$ is diassociative; $Q$ satisfies both left and right inverse properties; $Q$ satisfies both left and right alternative properties. Moreover, as we have already mentioned in the introduction, every diassociative automorphic loop is Moufang \cite{KiKuPh}. Thanks to Proposition \ref{Pr:AAIP}, we can refine these results as follows:

\begin{theorem}\label{Th:DAMoufang}
The following properties are equivalent for an automorphic loop $Q$:
\begin{enumerate}
\item[(i)] $Q$ has the left alternative property
\item[(ii)] $Q$ has the right alternative property,
\item[(iii)] $Q$ has the left inverse property,
\item[(iv)] $Q$ has the right inverse property,
\item[(v)] $Q$ is diassociative,
\item[(vi)] $Q$ is Moufang.
\end{enumerate}
\end{theorem}
\begin{proof}
Suppose that $Q$ has the left alternative property. Then Proposition \ref{Pr:AAIP} implies that $(yx\cdot x)^{-1} = x^{-1}\cdot x^{-1}y^{-1} = x^{-1}x^{-1}\cdot y^{-1} = (y\cdot xx)^{-1}$, so $Q$ has the right alternative property. A similar argument finishes the equivalence of (i) and (ii), and also proves the equivalence of (iii) and (iv). The rest follows from \cite{BrPa, KiKuPh}.
\end{proof}

We conclude this section with Bruck's proof of the fact that commutative Moufang loops are automorphic. The argument is based on nice observations about autotopisms and companions of pseudo-automorphisms, which we review.

Let $Q$ be a loop. A triple $(f,g,h)$ of bijections $Q\to Q$ is an \emph{autotopism} if $f(x)g(y)=h(xy)$ holds for every $x$, $y\in Q$. It is easy to see that the coordinate-wise product (composition) of autotopisms is an autotopism.

If a bijection $f$ of $Q$ and $c\in Q$ satisfy the identity $f(x)\cdot f(y)c = f(xy)c$, then $f$ is called a \emph{pseudo-automorphism} of $Q$ with \emph{companion} $c$.

\begin{lemma}[compare {{\cite[Lemma VII.2.1]{Bruck}}}]\label{Lm:PrincipalAtp}
Let $Q$ be a loop and $(f,g,h)$ an autotopism of $Q$ such that $f(1)=1$. Then $g=h$ and $g(x)=f(x)c$, where $c=g(1)$. Hence $f$ is a pseudo-automorphism with companion $c=g(1)$.
\end{lemma}
\begin{proof}
We have $g(x)=f(1)g(x)=h(1\cdot x)=h(x)$, so $g=h$. Also, $f(x)c = f(x)g(1)=h(x) = g(x)$. Finally, $f(x)\cdot f(y)c = f(x)g(y) = h(xy) = g(xy) = f(xy)c$.
\end{proof}

\begin{proposition}[{{\cite[Lemma VII.3.3]{Bruck}}}]\label{Pr:CML}
Commutative Moufang loops are automorphic.
\end{proposition}
\begin{proof}
Let $Q$ be a commutative Moufang loop. Let $f$ be a pseudo-automorphism of $Q$ with companion $c$. Then $f(x)\cdot cf(y) = f(x)\cdot f(y)c = f(xy)c = f(yx)c = f(y)\cdot f(x)c = f(x)c\cdot f(y)$ for every $x$, $y\in Q$, so $c\in N_m(Q)$. Since $Q$ is an inverse property loop, its nuclei coincide by Proposition \ref{Pr:IPNuclei} and we have $c\in N_r(Q)$. Then $c$ can be canceled in $f(x)\cdot f(y)c = f(xy)c$ and $f\in\aut{Q}$ follows.

It therefore suffices to prove that the mappings $L_{x,y}$ are pseudo-automorphisms. The Moufang identity \eqref{Eq:Moufang} is equivalent to the statement that $\varphi_x = (L_x,R_x,R_xL_x)$ is an autotopism of $Q$. Then $\varphi_{yx}^{-1}\varphi_y\varphi_x$ is an autotopism with first component $L_{x,y}$. By Lemma \ref{Lm:PrincipalAtp}, $L_{x,y}$ is a pseudo-automorphism.
\end{proof}

\section*{Lecture 2: Associated operations}
\setcounter{section}{2}
\setcounter{theorem}{0}
\setcounter{subsection}{0}

Many of the concepts presented in this section can be traced to two influential papers \cite{Glauberman1, Glauberman2} on loops of odd order written by Glauberman in the 1960s. In his study of Moufang loops $(Q,\cdot)$ of odd order \cite{Glauberman2}, the most important idea was to associate another loop $(Q,\bullet)$ with $(Q,\cdot)$, defined by $x\bullet y = x^{1/2}yx^{1/2}$, where $x^{1/2}$ is the unique square root of $x$ in $(Q,\cdot)$. The resulting loop $(Q,\bullet)$ is an instance of what would nowadays be called a Bruck loop (or a $K$-loop). This made Glauberman study Bruck loops of odd order and their left multiplication groups in detail \cite{Glauberman1} and establish a number of key results for them (see Theorem \ref{Th:Glauberman}). He then transferred the results from Bruck loops to Moufang loops.

We follow a similar approach but in a more general setting of twisted subgroups. We show how to associate left Bruck loops with uniquely $2$-divisible left Bol loops and with uniquely $2$-divisible automorphic loops. We then follow \cite{Greer} and establish a one-to-one correspondence between left Bruck loops of odd order and a certain class of commutative loops containing commutative automorphic loops of odd order. This will allow us to prove an analog of Theorem \ref{Th:Glauberman} for commutative automorphic loops. Finally, as in \cite{KiKuPhVo} we establish a one-to-one correspondence between uniquely $2$-divisible automorphic loops whose associated left Bruck loop is associative and a certain class of uniquely $2$-divisible Lie rings. This eventually leads to the Odd Order Theorem for automorphic loops. For the convenience of the reader, the correspondence results are visualized in Figure \ref{Fg:Correspondences}.

\subsection{Bruck loops}

A loop $Q$ is a \emph{left Bol loop} if it satisfies the \emph{left Bol identity}
\begin{equation}\label{Eq:LBol}
    x(y(xz)) = (x(yx))z.
\end{equation}
It is well known that left Bol loops have the left inverse property.

The following result gives a nice axiomatization of left Bol loops in the variety of magmas with inverses.

\begin{lemma}[{{\cite[(3.10)]{Kiechle}}} and {{\cite[Theorem 4.1]{PhVoScoop}}}]\label{Lm:LBolGroupoid}
Let $(Q,\cdot)$ be a groupoid with an identity element and two-sided inverses satisfying \eqref{Eq:LBol}. Then $(Q,\cdot)$ is a left Bol loop.
\end{lemma}

Consequently, a nonempty subset of a left Bol loop is a subloop if it is closed under mutiplication and inverses.

A \emph{left Bruck loop} is a left Bol loop with the \emph{automorphic inverse property} $(xy)^{-1}=x^{-1}y^{-1}$.

Here is an omnibus result on Bruck loops of odd order compiled from \cite{Glauberman1, Glauberman2}. Recall that the \emph{left multiplication group} of $Q$ is defined by $\lmlt{Q} = \genof{L_x}{x\in Q}$.

\begin{theorem}[Glauberman]\label{Th:Glauberman}
Let $Q$ be a left Bruck loop of odd order. Then $Q$ is solvable. If $H\le Q$ then $|H|$ divides $|Q|$. If $p$ is a prime dividing $|Q|$ then there is $x\in Q$ such that $|x|=p$. Sylow $p$-subloops and Hall $\pi$-subloops of $Q$ exist. The left multiplication group $\lmlt{Q}$ of $Q$ is of odd order.

If also $|Q|=p^k$ for an odd prime $p$, then $Q$ is centrally nilpotent.
\end{theorem}

\subsection{Twisted subgroups}

A subset $S$ of a group $G$ is a \emph{twisted subgroup} of $G$ if it contains the identity element of $G$, is closed under inverses, and is closed under the binary operation $(x,y)\mapsto xyx$.

Note that a twisted subgroup is not necessarily a subgroup, but every twisted subgroup $S$ is closed under powers. Indeed, it suffices to show that all positive powers of $x\in S$ belong to $S$, and we get this by induction on $k$ from $x^{k+2} = xx^k x$.

Call a subset $U$ of a loop $Q$ \emph{uniquely $2$-divisible} if the squaring map $Q\to Q$, $x\mapsto x^2$ restricts to a bijection on $U$. In this case, for every $x\in U$ there is a unique element $x^{1/2}\in U$ such that $(x^{1/2})^2 = x$. If $U$ happens to be power associative and $x\in U$ has odd order $n$, then $x^{1/2} = x^{(n+1)/2}$, so the square root of $x$ is a positive power of $x$. If $U$ happens to be closed under inverses, then $((x^{-1})^{1/2})^2 = x^{-1} = (x^{1/2}x^{1/2})^{-1} = (x^{1/2})^{-1} (x^{1/2})^{-1} = ((x^{1/2})^{-1})^2$ shows that $(x^{-1})^{1/2}$ is equal to $(x^{1/2})^{-1}$.

\begin{proposition}[compare {{\cite[Lemma 3]{Glauberman1}}}]\label{Pr:AssociatedBruck}
Let $G$ be a group and $S$ a uniquely $2$-divisible twisted subgroup of $G$. Then $(S,\circ)$ with multiplication
\begin{displaymath}
    x\circ y = (xy^2x)^{1/2}
\end{displaymath}
is a left Bruck loop. Moreover, the powers in $(S,\cdot)$ and $(S,\circ)$ coincide.
\end{proposition}
\begin{proof}
If $x$, $y\in S$ then $y^2\in S$, $xy^2x\in S$ and $(xy^2x)^{1/2}\in S$. Hence $(S,\circ)$ is a groupoid. Since $1\circ x = x = x\circ 1$ and $x^{-1}\circ x = (x^{-1}x^2x^{-1})^{1/2} = 1 = (xx^{-2}x)^{1/2} = x\circ x^{-1}$, we see that $(S,\circ)$ has identity element $1$ and two-sided inverses. Note that $x\circ (y\circ x) = (xyx^2yx)^{1/2} = ((xyx)^2)^{1/2} = xyx$. Thus $x\circ (y\circ (x\circ z)) = (xyxz^2xyx)^{1/2} = (xyx)\circ z = (x\circ (y\circ x))\circ z$. By Lemma \ref{Lm:LBolGroupoid}, $(S,\circ)$ is a left Bol loop in which inverses coincide with those of $(S,\cdot)$. It is a left Bruck loop thanks to $(x\circ y)^{-1} = ((xy^2x)^{1/2})^{-1} = ((xy^2x)^{-1})^{1/2} = (x^{-1}y^{-2}x^{-1})^{1/2} = x^{-1}\circ y^{-1}$. The inductive step $x\circ x^{n+1} = (xx^{2n+2}x)^{1/2} = x^{n+2}$ shows that powers in $(S,\cdot)$ and $(S,\circ)$ coincide.
\end{proof}

A twisted subgroup of a uniquely $2$-divisible group need not be uniquely $2$-divisible (consider $\mathbb Z$ in $(\mathbb Q,+)$). But note that if $G$ is a group of odd order then any twisted subgroup $S$ of $G$ is uniquely $2$-divisible.

The next result shows that in many varieties of loops the concepts ``uniquely $2$-divisible'' and ``of odd order'' coincide for finite loops.

\begin{lemma}\label{Lm:OddDiv}
Let $Q$ be a finite power-associative loop in which $|x|$ divides $|Q|$ for every $x\in Q$. Then the following conditions are equivalent:
\begin{enumerate}
\item[(i)] $Q$ is uniquely $2$-divisible,
\item[(ii)] $|Q|$ is odd,
\item[(iii)] $|x|$ is odd for every $x\in Q$.
\end{enumerate}
\end{lemma}
\begin{proof}
Condition (ii) implies (iii) by the assumption that $|x|$ divides $|Q|$. Conversely, if (iii) holds then the inversion mapping $x\mapsto x^{-1}$ is an involution with a unique fixed point $x=1$, so $|Q|$ is odd.

If (i) holds then $x^2=1$ implies $x=1$, so (iii) holds. Conversely, if (iii) holds, then $|x|=2n+1$ implies $(x^{n+1})^2 = x^{2n+2}=x$, so the squaring map is onto $Q$. Thanks to finiteness of $Q$, it is also one-to-one, and (i) follows.
\end{proof}

\subsection{Bruck loops associated with Bol and automorphic loops}

If $G$ is a uniquely $2$-divisible group, Proposition \ref{Pr:AssociatedBruck} with $S=G$ yields a uniquely $2$-divisible left Bruck loop $(G,\circ)$, the \emph{(left) Bruck loop associated with $G$}.

Proposition \ref{Pr:AssociatedBruck} cannot be used directly to associate left Bruck loops with nonassociative loops $Q$. The trick is to work with a certain twisted subgroup $S$ of $\mlt{Q}$ instead and then project the operation $\circ$ from $S$ to $Q$. The classical example is that of uniquely $2$-divisible left Bol loops, which we recall in Proposition \ref{Pr:BolBruck}.

\begin{proposition}[\cite{FoKiPh}]\label{Pr:BolBruck}
Let $(Q,\cdot)$ be a left Bol loop. Then $L_Q = \setof{L_x}{x\in Q}$ is a twisted subgroup of $\lmlt{Q}$ satisfying
\begin{equation}\label{Eq:BolTwistedSubgroup}
    L_xL_yL_x = L_{x(yx)}.
\end{equation}

If $(Q,\cdot)$ is also uniquely $2$-divisible, then $L_Q$ is uniquely $2$-divisible and $(Q,\circ)$ with multiplication
\begin{equation}\label{Eq:BolBruck}
    x\circ y = (x(y^2x))^{1/2}
\end{equation}
is a left Bruck loop in which powers coincide with those of $(Q,\cdot)$. When $Q$ is finite then any subloop of $(Q,\cdot)$ is a subloop of $(Q,\circ)$.
\end{proposition}
\begin{proof}
We have $1=L_1\in L_Q$, $L_x^{-1} = L_{x^{-1}}\in L_Q$ thanks to the left inverse property, and \eqref{Eq:BolTwistedSubgroup} follows from \eqref{Eq:LBol}. Therefore $L_Q$ is a twisted subgroup of $\lmlt{Q}$. An easy induction with \eqref{Eq:BolTwistedSubgroup} shows that $L_x^n = L_{x^n}$ for every $n\ge 0$.

Suppose that $(Q,\cdot)$ is uniquely $2$-divisible. The mapping $Q\to L_Q$, $x\mapsto L_x$ is a bijection since $L_x(1)=x$. Since $(L_{x^{1/2}})^2 = L_{(x^{1/2})^2} = L_x$, it follows that $L_Q$ is uniquely $2$-divisible with $L_x^{1/2} = L_{x^{1/2}}$. By Proposition \ref{Pr:AssociatedBruck}, $(L_Q,\circ)$ with multiplication $L_x\circ L_y = (L_xL_y^2L_x)^{1/2} = L_{(x(y^2x))^{1/2}}$ is a left Bruck loop with powers coinciding with those of $\lmlt{Q}$.

We claim that $\varphi:(L_Q,\circ)\to (Q,\circ)$, $L_x\mapsto x$ is an isomorphism of loops. Indeed, $\varphi$ is clearly a bijection and $\varphi(L_x\circ L_y) = \varphi( L_{(x(y^2x))^{1/2}}) = (x(y^2x))^{1/2} = x\circ y = \varphi(L_x)\circ \varphi(L_y)$.

Finally, suppose that $Q$ is finite and $S\le (Q,\cdot)$. To show that $S$ is a subloop of $(Q,\circ)$, it suffices to prove that it is closed under inverses and under the multiplication $\circ$. The former is true because the inverses in $(Q,\cdot)$ and $(Q,\circ)$ coincide, and the latter is true because $(S,\cdot)$ is closed under $\cdot$ and square roots (being positive integral powers in the finite case).
\end{proof}

A twisted subgroup in $\mlt{Q}$ is harder to find for automorphic loops. For $x\in Q$ define
\begin{displaymath}
    P_x = R_xL_{x^{-1}}^{-1}.
\end{displaymath}
Note that in automorphic loops we have $P_x = L_{x^{-1}}^{-1}R_x$ by Proposition \ref{Pr:PowerAssociative}.

\begin{proposition}[{{\cite[Proposition 4.2]{KiKuPhVo}}}]\label{Pr:AutomorphicBruck}
Let $(Q,\cdot)$ be an automorphic loop. Then $P_Q = \setof{P_x}{x\in Q}$ is a twisted subgroup of $\mlt{Q}$ satisfying
\begin{equation}\label{Eq:AutomorphicTwistedSubgroup}
    P_xP_yP_x = P_{P_x(y)} = P_{(x^{-1}\ldiv y)x}.
\end{equation}

If $(Q,\cdot)$ is also uniquely $2$-divisible, then $P_Q$ is uniquely $2$-divisible and $(Q,\circ)$ with multiplication
\begin{equation}\label{Eq:AutomorphicBruck}
    x\circ y = ((x^{-1}\ldiv y^2)x)^{1/2} = (x^{-1}\ldiv y^2x)^{1/2}
\end{equation}
is a left Bruck loop in which powers coincide with those of $(Q,\cdot)$. When $Q$ is finite then any subloop of $(Q,\cdot)$ is a subloop of $(Q,\circ)$.
\end{proposition}
\begin{proof}
We have $1 = P_1\in P_Q$. Proposition \ref{Pr:PowerAssociative} and Lemma \ref{Lm:LRT} yield
\begin{displaymath}
    P_xP_{x^{-1}} = R_xL_{x^{-1}}^{-1}R_{x^{-1}}L_x^{-1} = L_{x^{-1}}^{-1}R_{x^{-1}}L_x^{-1}R_x = T_{x^{-1}}T_x=1,
\end{displaymath}
so $P_x^{-1} = P_{x^{-1}}\in P_Q$. The identity \eqref{Eq:AutomorphicTwistedSubgroup} is nontrivial; see \cite[Proposition 3.4]{KiKuPhVo} for a proof. Therefore $P_Q$ is a twisted subgroup of $\mlt{Q}$. An easy induction with \eqref{Eq:AutomorphicTwistedSubgroup} yields $P_x^n = P_{x^n}$ for every $n\ge 0$, using $P_x(x^i) = (x^{-1}\ldiv x^i)x = x^{i+2}$.

Suppose that $(Q,\cdot)$ is uniquely $2$-divisible. The mapping $Q\to P_Q$, $x\mapsto P_x$ is a bijection since $P_x(1)=x^2$. Since $P_{x^{1/2}}^2 = P_{(x^{1/2})^2} = P_x$, it follows that $P_Q$ is uniquely $2$-divisible with $P_x^{1/2} = P_{x^{1/2}}$. By Proposition \ref{Pr:AssociatedBruck}, $(P_Q,\circ)$ with multiplication $P_x\circ P_y = (P_xP_y^2P_x)^{1/2} = P_{((x^{-1}\ldiv y^2)x)^{1/2}}$ is a left Bruck loop with powers coinciding with those of $\mlt{Q}$. Note that $(x^{-1}\ldiv y)x = x^{-1}\ldiv yx$ by Proposition \ref{Pr:PowerAssociative}.

We conclude as in the proof of Proposition \ref{Pr:BolBruck}, using the bijection $P_x\mapsto x$.
\end{proof}

When $(Q,\cdot)$ is a uniquely $2$-divisible automorphic loop, we call $(Q,\circ)$ from Proposition \ref{Pr:AutomorphicBruck} the \emph{left Bruck loop associated with} $(Q,\cdot)$.

It is worth noting that in left Bol loops we have $x^{-1}\ldiv y^2 = xy^2$ thanks to the left inverse property. So, in left Bol loops, the operation \eqref{Eq:BolBruck} of Proposition \ref{Pr:BolBruck} coincides with the operation \eqref{Eq:AutomorphicBruck} of Proposition \ref{Pr:AutomorphicBruck}. But neither result is a special case of the other.

\medskip

We can now easily deduce Cauchy's and Lagrange's theorems for automorphic loops of odd order from Theorem \ref{Th:Glauberman}.

\begin{theorem}\label{Th:ACauchyLagrange}
Let $Q$ be an automorphic loop of odd order. If $S$ is a subloop of $Q$ then $|S|$ divides $|Q|$. If $p$ is a prime dividing $|Q|$ then $Q$ contains an element of order $p$.
\end{theorem}
\begin{proof}
Let $(Q,\circ)$ be the left Bruck loop associated with $Q$. If $S\le Q$ then $(S,\circ)\le (Q,\circ)$ by Proposition \ref{Pr:AutomorphicBruck}. By Theorem \ref{Th:Glauberman}, $|S|$ divides $|Q|$. Let $p$ be a prime dividing $|Q|$. Then there is $x\in (Q,\circ)$ of order $p$ by Theorem \ref{Th:Glauberman}. Because powers in $Q$ and $(Q,\circ)$ coincide, $x$ has also order $p$ in $Q$.
\end{proof}

\begin{corollary}\label{Cr:pAut}
Every automorphic loop of prime order is associative.
\end{corollary}

Note that we cannot easily use Proposition \ref{Pr:AutomorphicBruck} to obtain the Odd Order Theorem for automorphic loops from the Odd Order Theorem for Bruck loops, for instance. The difficulty lies in the fact that it is not clear how subloops of $(Q,\circ)$ are related to subloops of $(Q,\cdot)$.

\subsection{Correspondence with Bruck loops}

By Proposition \ref{Pr:AutomorphicBruck}, if $(Q,\cdot)$ is a uniquely $2$-divisible automorphic loop then $P_Q$ is a twisted subgroup of $\mlt{Q}$ satisfying \eqref{Eq:AutomorphicTwistedSubgroup}, which induces a left Bruck loop operation $(Q,\circ)$ by $x\circ y = (x^{-1}\ldiv y^2x)^{1/2}$. However, there exist distinct uniquely $2$-divisible automorphic loops with the same associated left Bruck loops, so it is not possible to find an inverse to the mapping $(Q,\cdot)\mapsto (Q,\circ)$.

In an attempt to find a correspondence between uniquely $2$-divisible left Bruck loops and some class of loops, Greer \cite{Greer} defined a technical variety of loops as follows.

A loop $Q$ is a \emph{$\Gamma$-loop} if it is commutative, has the automorphic inverse property, satisfies $L_xL_{x^{-1}} = L_{x^{-1}}L_x$ and $P_xP_yP_x = P_{P_x(y)}$. Note that the last condition is just \eqref{Eq:AutomorphicTwistedSubgroup}. By \cite[Theorem 3.5]{Greer}, $\Gamma$-loops are power-associative.

Figure \ref{Fg:BGA} gives a Venn diagram of intersections of the varieties of left Bol loops, automorphic loops and $\Gamma$-loops. Here is a full justification for the diagram. If $Q$ is an automorphic $\Gamma$-loop then it is a commutative automorphic loop; conversely, a commutative automorphic loop is certainly automorphic and it satisfies the automorphic inverse property by Proposition \ref{Pr:AAIP}, the relation $L_xL_{x^{-1}} = L_{x^{-1}}L_x$ by Proposition \ref{Pr:PowerAssociative}, and \eqref{Eq:AutomorphicTwistedSubgroup} by \cite[Proposition 3.4]{KiKuPhVo}. If $Q$ is left Bol and automorphic then the antiautomorphic inverse property implies that $Q$ is Moufang (and automorphic); the converse is trivial. If $Q$ is left Bol and a $\Gamma$-loop then it is a commutative Moufang loop. If $Q$ is Moufang and a $\Gamma$-loop then it is a commutative Moufang loop. Finally, a commutative Moufang loop is automorphic by Proposition \ref{Pr:CML}.

\begin{figure}
\begin{tikzpicture}[scale=2]
   \begin{scope}[fill opacity=0.3]
      \fill[red] (0.866,0.5) circle (1.4); % Gamma loops
      \fill[green] (-0.866,0.5) circle (1.4); % left Bol loops
      \fill[blue] (0,-1) circle (1.4); % automorphic loops
   \end{scope}
   \draw (0.866,0.5) circle (1.4);
   \node at (0.866*1.6,0.5*1.6) {$\Gamma$};
   \draw (-0.866,0.5) circle (1.4);
   \node at (-0.866*1.6,0.5*1.6) {left Bol};
   \draw (0,-1) circle (1.4);
   \node at (0,-1*1.6) {automorphic};
   \node at (0,0) {CML};
   \node at (-0.866*0.9,-0.5*0.9) {Mfg A};
   \node at (0.866*0.85,-0.5*0.85) {comm A};
   \node at (0,1*0.82) {$\emptyset$};
   %\draw (0,1.59) arc (70:246:0.695);
\end{tikzpicture}
\caption{Intersections among left Bol loops, automorphic loops and $\Gamma$-loops.}\label{Fg:BGA}
\end{figure}
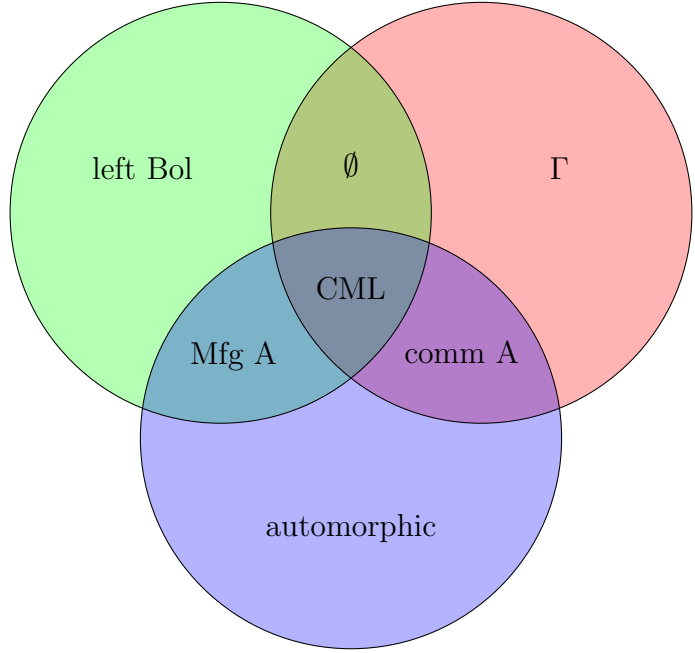

When $(Q,\cdot)$ is a uniquely $2$-divisible $\Gamma$-loop, we can use the same construction as in the case of uniquely $2$-divisible automorphic loops to obtain the \emph{associated left Bruck loop} $(Q,\circ)$, namely $x\circ y = (x^{-1}\ldiv y^2x)^{1/2}$. In the end, the variety of $\Gamma$-loops was chosen so that the proof of this result can mimic the proof in the automorphic case. (For instance, the difficult identity \eqref{Eq:AutomorphicTwistedSubgroup} is part of the definition of $\Gamma$-loops.) See \cite[Theorem 4.9]{Greer} for details.

Following Greer, we will now show how to construct a left Bruck loop $Q$ from a $\Gamma$-loop of odd order. (See the discussion after Lemma \ref{Lm:Greer2} for an obstacle in the more general uniquely $2$-divisible case.) We will actually use the twisted subgroup $L_Q$ again, but with a different operation.

On a uniquely $2$-divisible group $(G,\cdot)$, let
\begin{equation}\label{Eq:Greer}
    x*y = xy[y,x]^{1/2},
\end{equation}
where $[x,y] = x^{-1}y^{-1}xy$ is the usual commutator.

Straightforward, albeit nontrivial calculation with the commutator in groups yields:

\begin{lemma}[{{\cite[Theorem 2.5]{Greer}}}]\label{Lm:Greer1}
Let $(G,\cdot)$ be a uniquely $2$-divisible group. Then $(G,*)$ defined by \eqref{Eq:Greer} is a $\Gamma$-loop. Powers in $(G,\cdot)$ and $(G,*)$ coincide.
\end{lemma}

Let us now consider a twisted subgroup seemingly unrelated to $L_Q$; see \cite{Aschbacher, FoKiPh, Glauberman1}. For a group $G$ and $\tau\in\aut{G}$ let
\begin{displaymath}
    K(\tau) = \setof{x\in G}{\tau(x)=x^{-1}}.
\end{displaymath}
We claim that $K(\tau)$ is a twisted subgroup of $G$. Indeed, $1\in K(\tau)$ is clear, if $x\in K(\tau)$ then $\tau(x^{-1}) = \tau(x)^{-1} = (x^{-1})^{-1}$, so $x^{-1}\in K(\tau)$, and if $x$, $y\in K(\tau)$ then $\tau(xyx) = \tau(x)\tau(y)\tau(x) = x^{-1}y^{-1}x^{-1} = (xyx)^{-1}$, so $xyx\in K(\tau)$.

\begin{lemma}[compare {{\cite[Theorem 4.3]{FoKiPh}}}]\label{Lm:FoKiPh}
Let $G$ be a group and $\tau\in\aut{G}$. Let $S$ be a twisted subgroup of $G$ such that $S\subseteq K(\tau)$ and $\langle S\rangle = G$. Then $\setof{x^2}{x\in K(\tau)}\subseteq S$. In particular, if $G$ is a uniquely $2$-divisible group then $S=K(\tau)$.
\end{lemma}
\begin{proof}
Let $x\in K(\tau)$. Then $x^2 = x\tau(x^{-1})$. Since $\langle S\rangle = G$, there are $x_1$, $\dots$, $x_n\in S$ such that $x = x_1\cdots x_n$. Then $x\tau(x^{-1}) =  x_1\cdots x_n\tau(x_n^{-1}\cdots x_1^{-1}) = x_1\cdots x_n\tau(x_n^{-1})\cdots \tau(x_1^{-1}) = x_1\cdots x_nx_n\cdots x_1$, where we have used $x_i\in S\subseteq K(\tau)$. An easy induction on $n$ shows that the element $x_1\cdots x_nx_n\cdots x_1$ belongs to the twisted subgroup $S$.

We have proved $\setof{x^2}{x\in K(\tau)}\subseteq S\subseteq K(\tau)$. Suppose that $G$ is uniquely $2$-divisible. The squaring map is then injective on any twisted subgroup, and we claim that it is surjective on $K(\tau)$, so that $K(\tau)$ is uniquely $2$-divisible. Indeed, if $x\in K(\tau)$ then $\tau(x^{1/2}) = \tau(x)^{1/2} = (x^{-1})^{1/2} = (x^{1/2})^{-1}$, so $x^{1/2}\in K(\tau)$. It follows that $K(\tau) = \setof{x^2}{x\in K(\tau)}$, and $S=K(\tau)$.
\end{proof}

\begin{lemma}[compare {{\cite[Lemma 4.3]{Greer}}}]\label{Lm:Greer2}
Let $G$ be a uniquely $2$-divisible group and let $\tau\in\aut{G}$. Then $K(\tau)$ is a subloop of the $\Gamma$-loop $(G,*)$.
\end{lemma}
\begin{proof}
By Lemma \ref{Lm:Greer1}, $(G,*)$ is a $\Gamma$-loop. If $x$, $y\in K(\tau)$ then $\tau(x*y) = \tau(xy[y,x]^{1/2}) = \tau(x)\tau(y)[\tau(y),\tau(x)]^{1/2} = x^{-1}y^{-1}[y^{-1},x^{-1}]^{1/2} = x^{-1}*y^{-1} = (x*y)^{-1}$, where we have used the automorphic inverse property in the last step.

Let us now consider left division in $(G,*)$. The following statements are equivalent: $x*a=y$, $xa[a,x]^{1/2}=y$, $[a,x]=(a^{-1}x^{-1}y)^2$, $ax = ya^{-1}x^{-1}y$, $ay^{-1}a = x^{-1}yx^{-1}$, $(ay^{-1})^2 = x^{-1}yx^{-1}y^{-1}$, $a=(x^{-1}yx^{-1}y^{-1})^{1/2}y$. Since this is a term in $(G,\cdot)$, we can easily show that $K(\tau)$ is closed under left division in $(G,*)$.
\end{proof}

We would now like to apply Lemmas \ref{Lm:Greer1} and \ref{Lm:Greer2}. However, there are examples of uniquely $2$-divisible left Bruck loops $Q$ with $G=\lmlt{Q}$ not uniquely $2$-divisible, so the lemmas cannot be applied directly. We therefore focus on the odd case.

\begin{proposition}[\cite{Greer}]\label{Pr:Greer}
Let $(Q,\cdot)$ be a left Bruck loop of odd order and let $G=\lmlt{Q,\cdot}$. Then $(L_Q,*)$ is a $\Gamma$-loop, and $(Q,*)$ with multiplication
\begin{displaymath}
    x*y = (L_x*L_y)(1) = (L_xL_y[L_y,L_x]^{1/2})(1)
\end{displaymath}
is a $\Gamma$-loop.
\end{proposition}
\begin{proof}
Proposition \ref{Pr:BolBruck} shows that $L_Q$ is a twisted subgroup of $\lmlt{Q,\cdot}$. Let $\tau$ be the conjugation on $\sym{Q}$ by the inversion map $J$ of $(Q,\cdot)$. For $x$, $y\in Q$, we have $JL_xJ(y) = J(xy^{-1}) = x^{-1}y = L_{x^{-1}}(y) = L_x^{-1}(y)$ by the automorphic inverse property and the left inverse property. Because $\langle L_Q\rangle = G$, the established identity $\tau(L_x) = JL_xJ = L_{x^{-1}} = L_x^{-1}$ shows that $\tau\in\aut{G}$ and also that $L_Q\subseteq K(\tau)$.

By Theorem \ref{Th:Glauberman}, $|G|$ is odd. By Lemma \ref{Lm:OddDiv}, $G$ is uniquely $2$-divisible. Lemma \ref{Lm:FoKiPh} with $S=L_Q$ then gives $L_Q = K(\tau)$. By Lemma \ref{Lm:Greer2}, $(L_Q,*) = (K(\tau),*)$ is a subloop of the $\Gamma$-loop $(G,*)$. Finally, as usual, we transfer the operation $*$ from $(L_Q,*)$ to $(Q,*)$ by the isomorphism $L_x\mapsto x$.
\end{proof}

For a left Bruck loop $(Q,\cdot)$ of odd order, we call $(Q,*)$ from Proposition \ref{Pr:Greer} the \emph{$\Gamma$-loop associated with $(Q,\cdot)$}.

Greer went on to establish the announced one-to-one correspondence, and more:

\begin{theorem}[{{\cite[Theorem 5.2]{Greer}}}]\label{Th:GreerEquiv}
There is a categorical equivalence between left Bruck loops of odd order and $\Gamma$-loops of odd order. Given a $\Gamma$-loop $(Q,\cdot)$ of odd order, we let $(Q,\circ)$ be the associated left Bruck loop with multiplication $x\circ y = (x^{-1}\ldiv y^2x)^{1/2}$. Conversely, given a Bruck loop $(Q,\circ)$ of odd order, we let $(Q,\cdot)$ be the associated $\Gamma$-loop with multiplication $x\cdot y = (L_xL_y[L_y,L_x]^{1/2})(1)$, where $L_x$ is the left translation in $(Q,\circ)$.
\end{theorem}

Solvability, Lagrange and Cauchy theorems for commutative automorphic loops of odd order were for the first time established in \cite{JeKiVoStructure}. (See Theorems \ref{Th:EvenCauchyLagrange} and \ref{Th:EvenSolvable} for the even case.) The fact that commutative automorphic loops of odd order $p^k$ ($p$ a prime) are centrally nilpotent was proved independently in \cite{CsorgoNilp} and \cite{JeKiVoNilp}.

Theorem \ref{Th:GreerEquiv} allows us to obtain these and additional results from Glauberman's Theorem \ref{Th:Glauberman} not only for commutative automorphic loops of odd order but also for the larger class of $\Gamma$-loops of odd order.

\begin{theorem}[{{\cite[Section 6]{Greer}}}]\label{Th:Greer}
Let $Q$ be a $\Gamma$-loop of odd order. Then $Q$ is solvable and the Lagrange and Cauchy theorems hold for $Q$. Moreover, there are Sylow $p$- and Hall $\pi$-subloops in $Q$.

If also $|Q|=p^k$ for an odd prime $p$, then $Q$ is centrally nilpotent.
\end{theorem}

\subsection{Correspondence with Lie rings}

The correspondence between left Bruck loops of odd order and $\Gamma$-loops of odd order covered all commutative automorphic loops of odd order as a subclass of $\Gamma$-loops, but it did not cover all automorphic loops of odd order. In \cite{KiKuPhVo}, a one-to-one correspondence was found between uniquely $2$-divisible automorphic loops whose associated left Bruck loop is an abelian group on the one hand, and uniquely $2$-divisible Lie rings satisfying conditions \eqref{Eq:Wright1}, \eqref{Eq:Wright2} on the other hand (see Theorem \ref{Th:AutLie}). This partial correspondence is sufficient to establish the Odd Order Theorem for automorphic loops (Theorem \ref{Th:AOddOrderTheorem}). In this subsection we sketch the proofs of these results.

We start with a construction of Wright \cite{Wright}. Let us call $(Q,+,[.,.])$ an \emph{algebra} if $(Q,+)$ is a an abelian group and $[.,.]$ is biadditive, that is $[x+y,z] = [x,z]+[y,z]$ and $[x,y+z] = [x,y]+[x,z]$ for every $x$, $y$, $z\in Q$. In this situation, for every $x\in Q$ define
\begin{displaymath}
    \lad_x:Q\to Q,\,\lad_x(y) = [x,y],\quad \rad_x:Q\to Q,\,\rad_x(y)=[y,x]
\end{displaymath}
to be the \emph{left} and \emph{right adjoint maps of $x$}, respectively. Note that $\lad_x$, $\rad_x$ are just the left and right translations with respect to the binary operation $[.,.]$, respectively. Finally, for $x\in Q$ define
\begin{displaymath}
    \ell_x = \id_Q - \lad_x,\quad r_x = \id_Q -\rad_x.
\end{displaymath}

\begin{proposition}[see {{\cite[Proposition 8]{Wright}}} and {{\cite[Lemma 5.1]{KiKuPhVo}}}]\label{Pr:Wright}
Let $(Q,+,[.,.])$ be an algebra. Define a groupoid $(Q,\cdot)$ by
\begin{equation}\label{Eq:WrightDot}
    x\cdot y = x+y-[x,y].
\end{equation}
Then $(Q,\cdot)$ is a loop (necessarily with identity element $0$) if and only if
\begin{equation}\label{Eq:Wright1}
    \ell_x\text{ and }r_x\text{ are bijections of }Q
\end{equation}
for every $x\in Q$.

When $(Q,\cdot)$ is a loop with left and right translations $L_x$, $R_x$, respectively, then
\begin{displaymath}
    L_x(y) = x+\ell_x(y),\ R_x(y) = x + r_x(y),\ L_x^{-1}(y) = \ell_x^{-1}(y-x),\ R_x^{-1}(y) = r_x^{-1}(y-x).
\end{displaymath}
Moreover, $L_{x,y} = \ell_{yx}^{-1}\ell_y\ell_x$, $R_{x,y} = r_{xy}^{-1}r_yr_x$ and $T_x = \ell_x^{-1}r_x$.
\end{proposition}
\begin{proof}
We have $0\cdot x = x = x\cdot 0$ for every $x\in Q$. Note that $x\cdot y = x + \ell_x(y) = y + r_y(x)$. Hence $L_x$ bijects if and only if $\ell_x$ bijects, and $R_y$ bijects if and only if $r_y$ bijects.

The formulas for $L_x$, $R_x$, $L_x^{-1}$, $R_x^{-1}$ are straightforward. Let us calculate $L_{x,y}$. Note that every $\ell_x$ is additive, being a sum of two additive maps. We have
\begin{align*}
    L_{x,y}(z) &= L_{yx}^{-1}L_yL_x(z) = L_{yx}^{-1}L_y(x+\ell_x(z)) = L_{yx}^{-1}(y+\ell_y(x+\ell_x(z)))\\
        &= \ell_{yx}^{-1}(y + \ell_y(x) + \ell_y\ell_x(z)-yx) = \ell_{yx}^{-1}(yx + \ell_y\ell_x(z)-yx)\\
        &= \ell_{yx}^{-1}\ell_y\ell_x(z).
\end{align*}
Similarly for $R_{x,y}$ and $T_x$.
\end{proof}

Following Wright, we call $(Q,\cdot)$ the \emph{linear groupoid} of the algebra $(Q,+,[.,.])$, and the \emph{linear loop} of $(Q,+,[.,.])$ if \eqref{Eq:Wright1} holds. In view of Proposition \ref{Pr:Wright}, it is easy to express but difficult to understand in terms of properties of $[.,.]$ when the linear loop $(Q,\cdot)$ is automorphic. We therefore specialize to the setting of Lie rings.

An algebra $(Q,+,[.,.])$ is \emph{alternating} if $[x,x]=0$ for every $x\in Q$. Every alternating algebra is \emph{skew-symmetric}, that is, $[x,y]=-[y,x]$. (Proof: Expand $0=[x+y,x+y]$.)

We say that an algebra $(Q,+,[.,.])$ is \emph{uniquely $2$-divisible} if the abelian group $(Q,+)$ is uniquely $2$-divisible.

If $(Q,+,[.,.])$ is alternating, then $x\cdot x = x+x-[x,x]=2x$, so the associated linear groupoid is uniquely $2$-divisible if and only if $(Q,+,[.,.])$ is uniquely $2$-divisible.

A \emph{Lie ring} is an alternating algebra $(Q,+,[.,.])$ in which $[.,.]$ satisfies the \emph{Jacobi identity} $[x,[y,z]]+[y,[z,x]]+[z,[x,y]]=0$.

Even for Lie rings it is not easy to characterize when the associated linear loop is automorphic. We therefore analyze a stronger condition, namely $\ell_x$ and $r_x$ being automorphisms.

\begin{lemma}[compare {{\cite[Proposition 5.2]{KiKuPhVo}}}]
Let $(Q,+,[.,.])$ be a Lie ring and let $(Q,\cdot)$ be defined by \eqref{Eq:WrightDot}. Then $(Q,\cdot)$ is a loop and all mappings $\ell_x$, $r_x$ are automorphisms of $(Q,\cdot)$ if and only if conditions \eqref{Eq:Wright1} and
\begin{equation}\label{Eq:Wright2}
    [[x,Q],[x,Q]]=0
\end{equation}
hold for every $x\in Q$. In such a case, $(Q,\cdot)$ is automorphic.
\end{lemma}
\begin{proof}
By Proposition \ref{Pr:Wright}, $(Q,\cdot)$ is a loop if and only if \eqref{Eq:Wright1} holds. We therefore assume that \eqref{Eq:Wright1} holds and investigate when the bijections $\ell_x$, $r_x$ are automorphisms of $(Q,\cdot)$. Using skew-symmetry and the Jacobi identity freely, we have
\begin{align*}
    \ell_x(u)\ell_x(v) &= \ell_x(u)+\ell_x(v) - [\ell_x(u),\ell_x(v)]\\
        &= u-[x,u]+v-[x,v]-[u-[x,u],v-[x,v]]\\
        &= (u+v-[u,v]) - [x,u+v] + ([u,[x,v]]+[[x,u],v]) - [[x,u],[x,v]]\\
        &= (u+v-[u,v]) - [x,u+v] + [x,[u,v]] - [[x,u],[x,v]]\\
        &= (u+v-[u,v]) - [x,u+v-[u,v]] - [[x,u],[x,v]]\\
        &= uv - [x,uv] - [[x,u],[x,v]] = \ell_x(uv) - [[x,u],[x,v]].
\end{align*}
Therefore $\ell_x\in\aut{Q,\cdot}$ if and only if \eqref{Eq:Wright2} holds. The calculation for $r_x$ is similar.

By Proposition \ref{Pr:Wright}, $\inn{Q,\cdot}\le \genof{\ell_x,\,r_x}{x\in Q}$. Therefore, if $\ell_x$, $r_x\in\aut{Q,\cdot}$ for every $x\in Q$, the loop $(Q,\cdot)$ is automorphic.
\end{proof}

Our eventual goal is to prove the Odd Order Theorem for automorphic loops, so we focus on the uniquely $2$-divisible case.

\begin{lemma}\label{Lm:CirclePlus}
Let $(Q,+,[.,.])$ be a uniquely $2$-divisible Lie ring satisfying \eqref{Eq:Wright1} and \eqref{Eq:Wright2}. Let $(Q,\cdot)$ be the (uniquely $2$-divisible automorphic) linear loop of $(Q,+,[.,.])$. Let $(Q,\circ)$ be the (uniquely $2$-divisible) left Bruck loop associated with $(Q,\cdot)$. Then $(Q,\circ)=(Q,+)$ is an abelian group.
\end{lemma}
\begin{proof}
We have $x^2 = x+x-[x,x]=2x$, so $x^{1/2} = x/2$. Also, $x(-x) = x+(-x)+[x,-x] = 0$ shows $x^{-1}=-x$. Then $x\circ y = (x^{-1}\ldiv y^2x)^{1/2} = ((-x)\ldiv (2y)x)/2$. Therefore, the condition $x\circ y = x+y$ is equivalent to $(2y)x = (-x)\cdot (2(x+y))$, which is equivalent to $2y+x-[2y,x] = -x + 2(x+y) - [-x,2(x+y)]$, which follows easily because $[.,.]$ is alternating and biadditive.
\end{proof}

We have shown how to construct uniquely $2$-divisible automorphic loops from certain uniquely $2$-divisible Lie rings. In order to build a correspondence, we now need to return from uniquely $2$-divisible automorphic loops $(Q,\cdot)$ to Lie rings, i.e., we need to build operations $+$ and $[.,.]$ on $(Q,\cdot)$. Lemma \ref{Lm:CirclePlus} suggests to restrict our attention to the class of uniquely $2$-divisible automorphic loops whose associated left Bruck loop is an abelian group, and set $x+y=x\circ y$. This approach works. See \cite{KiKuPhVo} for a proof.

\begin{theorem}[{{\cite[Theorem 5.7]{KiKuPhVo}}}]\label{Th:AutLie}
Suppose that $(Q,+,[\cdot,\cdot])$ is a uniquely $2$-divisible Lie ring satisfying \eqref{Eq:Wright1} and \eqref{Eq:Wright2}. Then $(Q,\cdot)$ defined by \eqref{Eq:WrightDot} is a uniquely $2$-divisible automorphic loop whose associated left Bruck loop $(Q,\circ)$ is an abelian group (in fact, $(Q,\circ) = (Q,+)$).

Conversely, suppose that $(Q,\cdot)$ is a uniquely $2$-divisible automorphic loop whose associated left Bruck loop $(Q,\circ)$ is an abelian group. Then $(Q,\circ,[\cdot,\cdot])$ defined by
\begin{equation}\label{Eq:Bracket}
    [x,y] = x\circ y\circ (xy)^{-1}
\end{equation}
is a uniquely $2$-divisible Lie ring satisfying \eqref{Eq:Wright1} and \eqref{Eq:Wright2}.

Furthermore, the two constructions are inverse to one another. Subrings (resp. ideals) of the Lie ring are subloops (resp. normal subloops) of the corresponding automorphic loop, and subloops (resp. normal subloops) closed under square roots are subrings (resp. ideals) of the corresponding Lie ring.
\end{theorem}

Figure \ref{Fg:Correspondences} summarizes what we have learned so far. In the figure, all algebras are of odd order, left Bruck loops are blue, $\Gamma$-loops are red, automorphic loops are green, and Lie rings satisfying \eqref{Eq:Wright1} and \eqref{Eq:Wright2} are cyan. Dotted lines represent abelian groups. Automorphic loops whose associated left Bruck loops are associative are dashed green. Shaded regions represent one-to-one correspondences. Except for the associated operation $x\cdot y = L_xL_y[L_y,L_x]^{1/2}(1)$, all associated operations make sense in the uniquely $2$-divisible case, too.

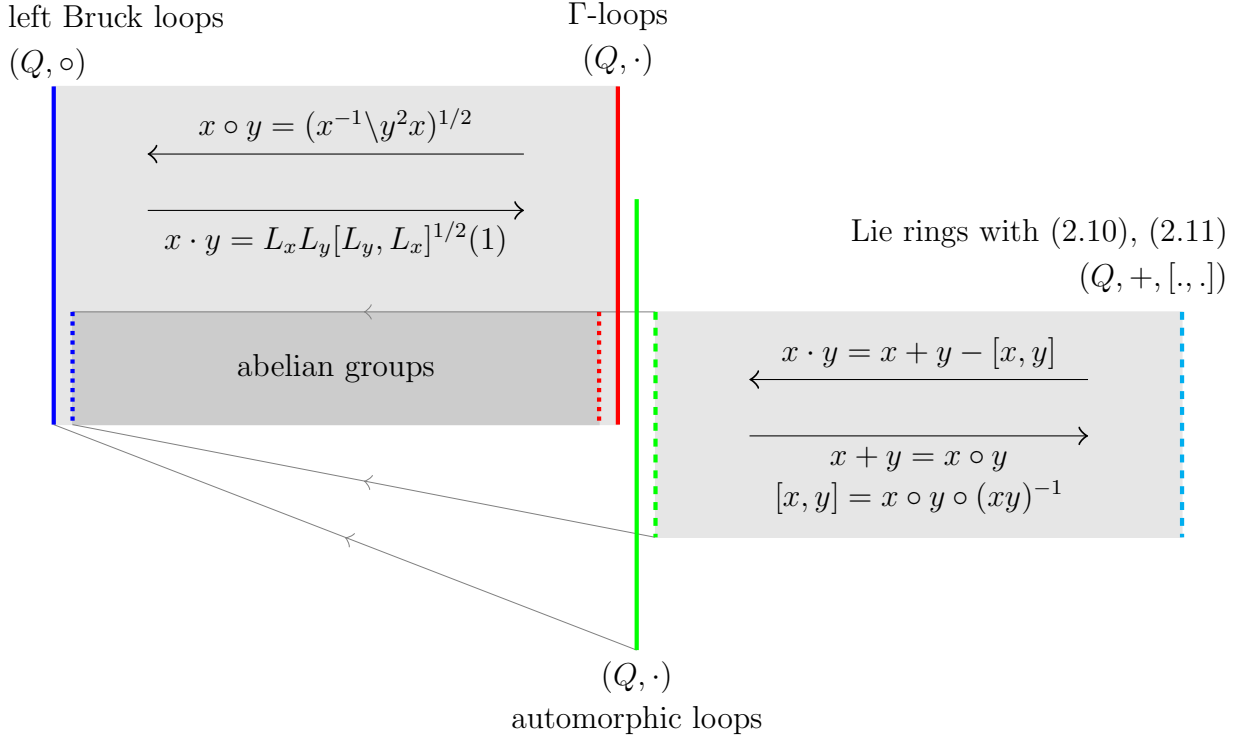
\begin{figure}
\tikzset{->-/.style={decoration={markings,mark=at position .5 with {\arrow[scale=2]{>}}},postaction={decorate}}}
\tikzset{-->/.style={decoration={markings,mark=at position 1 with {\arrow[scale=2]{>}}},postaction={decorate}}}
\begin{tikzpicture}[yscale = 0.15, xscale=0.25]
\draw [white!90!black, fill=white!90!black] (0,0) rectangle (30,-30);       % Bruck <-> Gamma
\draw [white!80!black, fill=white!80!black] (1,-20) rectangle (29,-30);     % abelian <-> abelian
\draw [->-, help lines] (32,-20) to (1,-20);                                % abelian <-> A-loops with abelian Bruck (not 1-1)
\draw [->-, help lines] (32,-40) -- (1,-30);
\draw [->-, help lines] (31,-50) -- (0,-30);                                % A-loops -> Bruck loops (not 1-1, not onto)
\draw [white!90!black, fill=white!90!black] (32,-20) rectangle (60,-40);    % A-loops with abelian Bruck <-> Lie rings with condition
\draw [blue, ultra thick] (0,0) -- (0,-30);                     % Bruck loops
\draw [blue, ultra thick, dotted] (1,-20) -- (1,-30);           % abelian groups in Bruck loops
\draw [red, ultra thick] (30,0) -- (30,-30);                    % Gamma loops
\draw [red, ultra thick, dotted] (29,-20) -- (29,-30);          % abelian groups in A-loops
\draw [green, ultra thick] (31,-10) -- (31,-50);                % A-loops
\draw [green, ultra thick, dashed] (32,-20) -- (32,-40);        % A-loops with abelian Bruck loops
\draw [cyan, ultra thick, dashed] (60,-20) -- (60,-40);         % Lie rings with conditions
\draw [-->] (25,-6) to (5,-6);
\node [above] at (15,-6) {$x\circ y = (x^{-1}\ldiv y^2x)^{1/2}$};
\draw [-->] (5,-11) to (25,-11);
\node [below] at (15,-11) {$x\cdot y = L_xL_y[L_y,L_x]^{1/2}(1)$};
\draw [-->] (55,-26) to (37,-26);
\node [above] at (46,-26) {$x\cdot y = x+y-[x,y]$};
\draw [-->] (37,-31) to (55,-31);
\node [below] at (46,-31) {$x+y=x\circ y$};
\node [below] at (46,-34) {$[x,y]=x\circ y\circ (xy)^{-1}$};
\node at (15,-25) {abelian groups};
\node [right] at (-3,2) {$(Q,\circ)$};
\node [right] at (-3,6) {left Bruck loops};
\node [above] at (30,0) {$(Q,\cdot)$};
\node [above] at (30,4) {$\Gamma$-loops};
\node [below] at (31,-50) {$(Q,\cdot)$};
\node [below] at (31,-54) {automorphic loops};
%\node [right] at (32,-38) {$(Q,\cdot)$};
%\node [right] at (32,-42) {automorphic with $(Q,\circ)$ abelian};
\node [left] at (63,-17) {$(Q,+,[.,.])$};
\node [left] at (63,-13) {Lie rings with \eqref{Eq:Wright1}, \eqref{Eq:Wright2}};
\end{tikzpicture}
\caption{Associated operations between left Bruck loops, $\Gamma$-loops, automorphic loops and Lie rings of odd order.}\label{Fg:Correspondences}
\end{figure}

We now work toward the Odd Order Theorem for automorphic loops.

\begin{lemma}[{{\cite[Lemma 5.8]{KiKuPhVo}}}]\label{Lm:2Solvable}
Let $(Q,+,[.,.])$ be a uniquely $2$-divisible Lie ring. Then \eqref{Eq:Wright2} holds if and only if $(Q,+,[.,.])$ is solvable of length $2$, that is, $[[Q,Q],[Q,Q]]=0$.
\end{lemma}

\begin{lemma}[{{\cite[Lemma 6.5]{KiKuPhVo}}}]\label{Lm:NormalInDot}
Let $(Q,\cdot)$ be an automorphic loop of odd order, let $(Q,\circ)$ be the associated left Bruck loop, and let $S$ be a characteristic subloop of $(Q,\circ)$. Then $S$ is a normal subloop of $(Q,\cdot)$.
\end{lemma}
\begin{proof}
Since $x\circ y = (x^{-1}\ldiv y^2x)^{1/2}$, we have $\aut{Q,\cdot}\le\aut{Q,\circ}$. Thus $S$ is invariant under $\inn{Q,\cdot}\le\aut{Q,\cdot}$. Let $u$, $v\in S$. We will show that $vu$ and $v\rdiv u\in S$. Let $w=v^{1/2}$. Since powers in $(Q,\cdot)$ and $(Q,\circ)$ coincide, $w\in S$. Then $T_u((u\circ w)^2) = (T_u(u\circ w))^2 = (T_u(u)\circ T_u(w))^2 = (u\circ T_u(w))^2 = u^{-1}\ldiv T_u(w)^2u = u^{-1}\ldiv T_u(v)u = L_{u^{-1}}^{-1}R_uT_u(v) = L_{u^{-1}}^{-1}L_u^{-1}R_u^2(v)$ is an element of $S$, where we have used Proposition \ref{Pr:PowerAssociative} in the last equality. Since $L_uL_{u^{-1}}\in\inn{Q,\cdot}$, it follows that $R_u^2(v)\in S$. By induction, $R_u^{2m}(v)\in S$ for every $m$. By Lemma \ref{Lm:OddDiv}, $|u|=2m+1$ for some $m$. Then $R_u^{2m+1}\in\inn{Q,\cdot}$, so also $R_u^{-2m}R_u^{2m+1}(v) = vu$ and $R_u^{-2m-2}R_u^{2m+1}(v)=v\rdiv u\in S$. By the antiautomorphic inverse property for $(Q,\cdot)$, $v\ldiv u\in S$, too.

We have shown that $S$ is a subloop of $(Q,\cdot)$. It is a normal subloop because $S$ is invariant under $\inn{Q,\cdot}$.
\end{proof}

\begin{theorem}[{{\cite[Theorem 6.6]{KiKuPhVo}}}]\label{Th:AOddOrderTheorem}
Automorphic loops of odd order are solvable.
\end{theorem}
\begin{proof}
Let $(Q,\cdot)$ be a minimal counterexample. If $S$ is a nontrivial, proper normal subloop of $(Q,\cdot)$ then, by minimality, both $S$ and $(Q,\cdot)/S$ are solvable automorphic loops of odd order. This contradicts the nonsolvability of $(Q,\cdot)$. Therefore $(Q,\cdot)$ is simple.

Let $(Q,\circ)$ be the associated left Bruck loop. By Theorem \ref{Th:Glauberman}, $(Q,\circ)$ is solvable and so the derived subloop $D=(Q,\circ)'$ is a proper subloop of $(Q,\circ)$. Since $D$ is a characteristic subloop of $(Q,\circ)$, Lemma \ref{Lm:NormalInDot} shows that $D$ is normal in $(Q,\cdot)$. Since $(Q,\cdot)$ is simple, $D = 1$ and $(Q,\circ)$ is an abelian group.

Recall that powers in $(Q,\cdot)$ and $(Q,\circ)$ agree. Let $p$ be a prime divisor of $|Q|$ and let $Q_p=\setof{x\in Q}{x^p=1}$. Then $Q_p$ is a characteristic subloop of $(Q,\circ)$, hence a normal subloop of $(Q,\cdot)$. By Theorem \ref{Th:ACauchyLagrange}, $Q_p$ is nontrivial, so $Q_p = Q$ because $(Q,\cdot)$ is simple. Thus $(Q,\cdot)$ has exponent $p$, $(Q,\circ)$ has exponent $p$, and $(Q,\circ)$ is an elementary abelian $p$-group.

By Theorem \ref{Th:AutLie}, $(Q,\circ,[\cdot,\cdot])$ defined by \eqref{Eq:Bracket} is a Lie ring satisfying \eqref{Eq:Wright1} and \eqref{Eq:Wright2}. By Lemma \ref{Lm:2Solvable}, $(Q,\circ,[\cdot,\cdot])$ is solvable (of class $2$). Since $(Q,\circ)$ is an elementary abelian $p$-group, we may view $(Q,\circ,[\cdot,\cdot])$ as a finite dimensional Lie algebra over $GF(p)$. Since $(Q,\cdot)$ is simple, Theorem \ref{Th:AutLie} also implies that $(Q,\circ,[\cdot,\cdot])$ is either a simple Lie algebra or an abelian Lie algebra (that is, $[Q,Q]=0$). The former case contradicts solvability of $(Q,\circ,[\cdot,\cdot])$, and so $(Q,\circ,[\cdot,\cdot])$ is an abelian Lie algebra. But then $x\cdot y = x\circ y\circ [x,y] = x\circ y$, so $(Q,\cdot)$ is an abelian group, a contradiction with nonsolvability of $(Q,\cdot)$.
\end{proof}

\section*{Lecture 3: Enumerations and constructions}
\setcounter{section}{3}
\setcounter{theorem}{0}
\setcounter{subsection}{0}

In this section we first show how to efficiently search for finite simple automorphic loops, temporarily suspending the notation $\circ$ and $*$ from previous sections. Then we discuss (commutative) automorphic loops of order $pq$ and $p^3$. Finally, we give two useful constructions of automorphic loops.

\subsection{Enumerating all left automorphic loops}

Let $G$ be a permutation group on a finite set $Q=\{1,\dots,d\}$, and let $H\le G$. The first goal of this section is to present a naive algorithm for constructing all loops $(Q,*)$ on $Q$ with identity element $1$ so that $\lmlt{Q,*}\le G$ and $H\le\aut{Q,*}$. Since $\lmlt{Q,*}$ acts transitively on $Q$ and $\varphi(1)=1$ holds for every $\varphi\in H$, let us assume from the start that $G$ is transitive on $Q$ and $H\le G_1$.

We then specialize this algorithm to construct all left automorphic loops $(Q,*)$ on $Q$ satisfying $\lmlt{Q,*}=G$. In the next subsection we will add the requirement that $(Q,*)$ be simple. The exposition follows \cite{JoKiNaVo}.

\begin{lemma}\label{Lm:LeftTranslations}
Let $Q=\{1,\dots,d\}$ be a finite set and let $L = \setof{\ell_x}{x\in Q}$ be a subset of $\sym{Q}$. Then $(Q,*)$ defined by $x*y = \ell_x(y)$ is a loop with identity element $1$ if and only if
\begin{enumerate}
\item[(i)] $\ell_1$ is the identity mapping on $Q$, and
\item[(ii)] $\ell_x(1)=x$ for every $x\in Q$, and
\item[(iii)] $\ell_x^{-1}\ell_y$ is fixed-point free for every $x$, $y\in Q$ with $x\ne y$.
\end{enumerate}
\end{lemma}
\begin{proof}
Condition (i) holds iff $x = \ell_1(x)= 1*x$ for every $x\in Q$. Condition (ii) hold iff $x = \ell_x(1) = x*1$ for every $x\in Q$. So (i) and (ii) together are equivalent to $(Q,*)$ having $1$ as the identity element. Since $L\subseteq\sym{Q}$, all the left translations of $(Q,*)$ are bijections. Let $z\in Q$. Then $z$ is not a fixed point of $\ell_x^{-1}\ell_y$ if and only if $x*z\ne y*z$. Therefore condition (iii) holds if and only if all right translations of $(Q,*)$ are one-to-one. We are done by finiteness of $Q$.
\end{proof}

We therefore have the following naive algorithm for constructing all loops on $Q$ with identity element $1$: Construct all subsets $\setof{\ell_x}{x\in Q}$ of $\sym{Q}$ and check that conditions (i)--(iii) of Lemma \ref{Lm:LeftTranslations} hold.

We will show how to speed up the algorithm if we are only interested in left automorphic loops, essentially by adding left translation not one at a time but rather one conjugacy class at a time.

\begin{lemma}\label{Lm:ConjClasses}
Let $Q$ be a loop.
\begin{enumerate}
\item[(i)] A bijection $\varphi:Q\to Q$ is an automorphism of $Q$ if and only if $\varphi L_x\varphi^{-1} = L_{\varphi(x)}$ for every $x\in Q$.
\item[(ii)] If $\varphi\in\aut{Q}$ fixes $x$ then $L_x\varphi = \varphi L_x$.
\end{enumerate}
\end{lemma}
\begin{proof}
The following conditions, universally quantified for $y\in Q$, are equivalent: $\varphi L_x \varphi^{-1} = L_{\varphi(x)}$, $\varphi(x\varphi^{-1}(y)) = \varphi(x)y$, $\varphi(xy) = \varphi(x)\varphi(y)$. To prove (ii), consider $\varphi\in\aut{Q}$ that fixes $x$, and note that $L_x\varphi(y) = x\varphi(y) = \varphi(x)\varphi(y) = \varphi(xy) = \varphi L_x(y)$ for every $y\in Q$.
\end{proof}

\begin{algorithm}\label{Al:Basic}\ \\
\noindent\emph{Input:} A set $Q=\{1,\dots,d\}$, a transitive permutation group $G$ on $Q$, and $H\le G_1$.\medskip

\noindent\emph{Output:} All loops $(Q,*)$ on $Q$ with identity element $1$ such that $\lmlt{Q,*}\le G$ and $H\le\aut{Q,*}$.\medskip

\noindent\emph{Step 1:} Let $\ell_1 = 1_G$, and let $X\subseteq Q\setminus\{1\}$ be a set of orbit representatives for the natural action of $H$ on $Q\setminus\{1\}$. (The condition $\ell_1=1_G$ is forced by Lemma \ref{Lm:LeftTranslations}(i).)\medskip

\noindent\emph{Step 2:} For all $x\in X$, let
\begin{displaymath}
    \mathcal L_x = \setof{\ell_x\in G}{\ell_x(1)=x,\,\ell_x\text{ is fixed-point free, and }\ell_x\in C_G(H_x)}.
\end{displaymath}
If $\mathcal L_x=\emptyset$, stop with failure. (This is a set of candidates for $\ell_x$. The first two conditions are necessary by Lemma \ref{Lm:LeftTranslations}. The last condition is necessary by Lemma \ref{Lm:ConjClasses}(ii). Note that if $\mathcal L_x$ is nonempty, it suffices to find one $\ell\in\mathcal L_x$ and set $\mathcal L_x = \ell( C_G(H_x)_1 )$.)\medskip

\noindent\emph{Step 3:} For all $x\in X$, let
\begin{displaymath}
    \mathbb L_x = \setof{\ell_x^H}{\ell_x\in\mathcal L_x,\,|\ell_x^H| = |H(x)|,\,\ell_x^{-1}\ell\text{ is fixed-point free for every }\ell\in\ell_x^H\text{ with }\ell\ne\ell_x}.
\end{displaymath}
If $\mathbb L_x = \emptyset$, stop with failure. (By Lemma \ref{Lm:ConjClasses}, the desired $L=\setof{\ell_x}{x\in Q}$ is a union of $H$-conjugacy classes in $G$. The set $\mathbb L_x$ is a set of candidates for the $H$-conjugacy class containing $\ell_x$. The condition $|\ell_x^H| = |H(x)|$ is forced by Lemma \ref{Lm:ConjClasses}(i). The second condition is forced by Lemma \ref{Lm:LeftTranslations}(iii).)\medskip

\noindent\emph{Step 4:} Construct a graph $\Gamma$ on $V=\bigcup_{x\in X}\mathbb L_x$ by letting $(\ell_x^H,\ell_y^H)\in\mathbb L_x\times\mathbb L_y$ to be an edge if and only if $(\ell_x^H)^{-1}(\ell_y^H)$ consists of fixed-point free permutations. (Note that it suffices to check that $\ell_x^{-1}\ell_y^H$ consists of fixed-point free permutations. Indeed, if $\psi\ell_x\psi^{-1}(z) = \varphi\ell_y\varphi^{-1}(z)$ for some $z\in Q$, then $\ell_x(\psi^{-1}(z)) = (\psi^{-1}\varphi)\ell_y(\psi^{-1}\varphi)^{-1}(\psi^{-1}(z))$.)\medskip

\noindent\emph{Step 5:} Find all subsets $C$ of $V$ such that $C$ is a clique in $\Gamma$ and $\sum_{v\in C}|v| = |Q|-1$. If there are no such $C$, stop with failure. Else return all loops $Q(L) = (Q,*)$, where $L=L(C)=\{\ell_1\}\cup \bigcup_{v\in C}v = \setof{\ell_x}{x\in Q}$ and $x*y = \ell_x(y)$. (The clique property accounting for $|Q|-1$ left translations is at this stage necessary and sufficient by Lemmas \ref{Lm:LeftTranslations} and \ref{Lm:ConjClasses}.)
\end{algorithm}

\medskip

Denote by $\mathcal A_\ell^\le(Q,G)$ all left automorphic loops $(Q,*)$ defined on $Q$ with identity element $1$ and satisfying $\lmlt{Q,*}\le G$, by $\mathcal A_\ell^=(Q,G)$ all loops $(Q,*)\in \mathcal A_\ell^\le(Q,G)$ with $\lmlt{Q,*}=G$, and by $\mathcal A^=(Q,G)$ all loops $(Q,*)\in\mathcal A_\ell^\le(Q,G)$ that are automorphic and satisfy $\mlt{Q,*}=G$. Let also $\mathcal C(Q,G,H)$ be the set of all loops $(Q,*)$ obtained by Algorithm \ref{Al:Basic} with input $Q$, $G$ and $H$.

\begin{lemma}\label{Lm:LeftAlgorithm}
Let $G$ be a transitive permutation group on $Q=\{1,\dots,d\}$. Then $\mathcal A_\ell^=(Q,G) \subseteq \mathcal C(Q,G,G_1) \subseteq \mathcal A_\ell^\le(Q,G)$. Moreover, $\mathcal A^=(Q,G)\subseteq C(Q,G,G_1)$.
\end{lemma}
\begin{proof}
First let $(Q,*)\in \mathcal A_\ell^=(Q,G)$. Then $\linn{Q,*} = \lmlt{Q,*}_1 = G_1$, and therefore $(Q,*)\in \mathcal C(Q,G,G_1)$. Now let $(Q,*)\in C(Q,G,G_1)$. Then $\lmlt{Q,*}\le G$ because every left translation of $(Q,*)$ is in $G$. Since $\linn{Q,*} = \lmlt{Q,*}_1\le G_1\le\aut{Q,*}$, the loop $(Q,*)$ is left automorphic. Finally, let $(Q,*)\in \mathcal A^=(Q,G)$. Then $\lmlt{Q,*}\le G$ and $G_1=\mlt{Q,*}_1 = \inn{Q,*}\le\aut{Q,*}$. Thus $(Q,*)\in\mathcal C(Q,G,G_1)$.
\end{proof}

Lemma \ref{Lm:LeftAlgorithm} can be used to find all left automorphic loops on the set $Q=\{1,\dots,d\}$ with identity element $1$. It suffices to apply the lemma to all transitive groups $G$ in $Q$ and discard duplicate loops.

\subsection{Searching for finite simple automorphic loops}

Recall that a loop $Q$ is said to be \emph{simple} if it has no normal subloops except for $Q$ and $1$.

In principle, Algorithm \ref{Al:Basic} returns all finite left automorphic loops, and hence also all finite simple automorphic loops. In practice, the algorithm is too slow to get to even moderately large orders. In this section we will describe improvements to the algorithm so that it can check for simple automorphic loops of order up to several thousands.

The key results are due to Albert and Vesanen. Albert's result is easy to prove, Vesanen's not so much.

\begin{theorem}[{{\cite[Theorem 8]{Albert}}}]
A loop $Q$ is simple if and only if its multiplication group $\mlt{Q}$ acts primitively on $Q$.
\end{theorem}

\begin{theorem}[\cite{Vesanen}]\label{Th:Vesanen}
Let $Q$ be a finite loop. If $\mlt{Q}$ is solvable then $Q$ is solvable.
\end{theorem}

Recall that a partition of $Q$ is said to be \emph{trivial} if it is of the form $\{Q\}$ or of the form $\setof{\{x\}}{x\in Q}$. A group $G\le\sym{Q}$ \emph{preserves} a partition $\{B_1,\dots,B_n\}$ of $Q$ if for every $\varphi\in G$ and every $1\le i\le n$ there is $1\le j\le n$ such that $\varphi(B_i)=B_j$. A transitive permutation group $G\le\sym{Q}$ is \emph{primitive} if it does not preserve any nontrivial partition of $Q$. The \emph{degree} of $G$ is the cardinality of $Q$.

It is easy to see that every $2$-transitive group $G\le\sym{Q}$ is primitive. (Consider a nontrivial partition $\{B_1,\dots,B_n\}$ with $n\ge 1$, $B_1$ containing distinct elements $x$, $y$, and let $z\in B_2$. Let $\varphi\in G$ be such that $\varphi(x)=x$ and $\varphi(y)=z$. Then $\varphi(B_1)\ne B_j$ for every $1\le j\le n$.) Unlike finite $2$-transitive groups, finite primitive groups are not classified \cite{DiMo}. \texttt{GAP} contains a library of all primite groups of degree $<2500$. \texttt{MAGMA} \cite{BoCaPl} contains a library of all primitive groups of degree $<4096$.

\begin{lemma}\label{Lm:3Transitive}
If $Q$ is a loop of order bigger than $4$ and $H\le\aut{Q}$ then $H$ is not $3$-transitive on $Q\setminus\{1\}$.
\end{lemma}
\begin{proof}
Suppose that $H$ is $3$-transitive on $Q\setminus\{1\}$. Let $x$, $y\in Q$ be such that $|\{1$, $x$, $y\}|=3$ and $z=xy\ne 1$. Then $\{x,y,z\}$ is a subset of $Q\setminus\{1\}$ of cardinality $3$. Let $\varphi\in H$ be such that $\varphi(x)=x$, $\varphi(y)=y$ and $\varphi(z)\ne z$. (Here we use $|Q|>4$.) We reach a contradiction with $\varphi(z) = \varphi(xy) = \varphi(x)\varphi(y)=xy=z$.
\end{proof}

\begin{proposition}\label{Pr:BetterAlg}
All finite simple nonassociative automorphic loops are found in the set $\bigcup \mathcal C(Q,G,G_1)$, where the union is taken over sets $Q$ of even order and over primitive groups $G\le\sym{Q}$ that are not solvable and not $4$-transitive.
\end{proposition}
\begin{proof}
Let $(Q,*)$ be a finite simple nonassociative automorphic loop of order $d>1$ with the identity element $1$. Let $G=\mlt{Q,*}$. If $(Q,*)$ is solvable then it is an abelian group, a contradiction. By Theorem \ref{Th:AOddOrderTheorem}, we can assume that $d$ is even. By Theorem \ref{Th:Vesanen}, $G$ is not solvable. If $G$ is $4$-transitive, then $G_1\le\aut{Q,*}$ is $3$-transitive on $Q\setminus\{1\}$, a contradiction with Lemma \ref{Lm:3Transitive}. It remains to show that $(Q,*)\in\mathcal C(Q,G,G_1)$. This follows from Lemma \ref{Lm:LeftAlgorithm}.
\end{proof}

Let $(Q,*)\in \bigcup \mathcal C(Q,G,G_1)$, where the union is as in Proposition \ref{Pr:BetterAlg}. Suppose that we run the algorithm by incrementally increasing the cardinality of $Q$, and, for a fixed $d=|Q|$, by incrementally increasing the order of $G$. When should we catalog $(Q,*)$ as a newly found finite simple nonassociative automorphic loop? We first calculate the order of $M=\mlt{Q,*}\le G$. If $|M|<|G|$ then $(Q,*)$ is guaranteed to be automorphic (since $\inn{Q,*}=M_1\le G_1\le\aut{Q,*}$) but either $M$ is not as in Proposition \ref{Pr:BetterAlg} or we have already seen $(Q,*)$ in $\mathcal C(Q,M,M_1)$, so we do not store $(Q,*)$. If $|M|>|G|$ then $(Q,*)$ is either not automorphic (checking this is expensive), or we will see the same loop later in $\mathcal C(Q,M,M_1)$, so we again do not store it. If $|M|=|G|$ then $(Q,*)$ is a finite simple nonassociative automorphic loop and we store it (upon checking for isomorphism against all already stored loops with the same multiplication group).

This search has been carried out in \cite{JoKiNaVo} for $d<2500$ and recently by Cameron and Leemans \cite{CaLe} for $d<4096$. The result is somewhat surprising:

\begin{proposition}
There are no finite simple nonassociative automorphic loops of order less than $4096$.
\end{proposition}

We remark that Algorithm \ref{Al:Basic} finds numerous finite simple nonassociative left automorphic loops.

\medskip

Are there any finite simple nonassociative commutative automorphic loops? The search for finite simple commutative automorphic loops can be reduced to orders $2^k$ by the following result (whose proof, incidentally, required another associated operation to show that a product of two squares is a square):

\begin{theorem}[\cite{JeKiVoStructure}]\label{Th:Decomposition}
Let $Q$ be a finite commutative automorphic loop. Then $Q$ is a direct product $A\times B$, where $A=\setof{x\in Q}{|x| = 2^n\text{ for some }n}$ and $B=\setof{x\in Q}{|x|\text{ is odd}}$. Morever, $|A|=2^m$ for some $m$ and $|B|$ is odd.
\end{theorem}

With this decomposition at hand, we easily get:

\begin{theorem}[\cite{JeKiVoStructure}]\label{Th:EvenCauchyLagrange}
Let $Q$ be a finite commutative automorphic loop. Then the Cauchy and Lagrange theorems hold for $Q$.
\end{theorem}

It is much harder to deduce solvability in the even case. Grishkov, Kinyon and Nagy used advanced results on Lie algebras to prove:

\begin{theorem}[\cite{GrKiNa}]\label{Th:EvenSolvable}
Every finite commutative automorphic loop is solvable.
\end{theorem}

Thus there are no finite simple nonassociative commutative automorphic loops.

\subsection{Commutative automorphic loops of order $pq$}

Recall that a power-associative loop $Q$ is a $p$-loop if every element of $Q$ has order that is a power of $p$. From Theorem \ref{Th:EvenCauchyLagrange} we easily deduce that, for an odd prime $p$, a finite automorphic loop is a $p$-loop if and only if $|Q|$ is a power of $p$.

Let us now consider finite commutative automorphic loops. Unlike in abelian groups, the direct factor $B$ from Theorem \ref{Th:Decomposition} does not necessarily decompose as a direct product of $p$-loops. In fact, for certain odd primes $p>q$, Dr\'apal constructed a nonassociative commutative automorphic loop $Q$ of order $pq$, which therefore does not factor as a direct product of an automorphic loop of order $p$ and an automorphic loop of order $q$. We will discuss his construction at the end of this subsection. First we have a look at commutative automorphic loops of order $pq$ in general.

\begin{lemma}\label{Lm:NonCyclic}
Let $Q$ be a power-associative loop. Then $Q/Z(Q)$ is never a nontrivial cyclic group.
\end{lemma}
\begin{proof}
Suppose that $Q/Z(Q)$ is cyclic of order $m>1$. Then there is $x\in Q\setminus Z(Q)$ such that $xZ(Q)$ has order $m$ in $Q/Z(Q)$ and $Q = \bigcup_{0\le i<m} x^iZ(Q)$. Therefore any element of $Q$ can be written as $x^ia$ for some $0\le i<m$ and $a\in Z(Q)$. With three elements of $Q$ written in this form, we calculate
\begin{displaymath}
    (x^ia\cdot x^jb)\cdot x^kc = (x^ix^j)x^k\cdot abc = x^i(x^jx^k)\cdot abc = x^ia\cdot (x^jb\cdot x^kc),
\end{displaymath}
where we have used $a$, $b$, $c\in Z(Q)$ and power-associativity for $\langle x\rangle$. Hence $Q$ is a group, and the result follows from the well-known fact that, in groups, $Q/Z(Q)$ is never a nontrivial cyclic group.
\end{proof}

Niederreiter and Robinson proved the following result while studying Bol loops of order $pq$:

\begin{proposition}[\cite{NiRo}]\label{Pr:NR}
Let $Q$ be a left Bol loop of order $pq$ with odd primes $p>q$. Then $Q$ contains a unique subloop of order $p$.
\end{proposition}

\begin{lemma}\label{Lm:pq}
Let $Q$ be a nonassociative commutative automorphic loop of order $pq$ with odd primes $p>q$. Then $Z(Q)=1$, $Q$ contains a normal subgroup $S$ of order $p$, and all elements of $Q\setminus S$ have order $q$.
\end{lemma}
\begin{proof}
We have $Z(Q)<Q$ by assumption. If $1<Z(Q)$ then $Q/Z(Q)$ is isomorphic to $\mathbb Z_p$ or to $\mathbb Z_q$ by Corollary \ref{Cr:pAut}, a contradiction with Lemma \ref{Lm:NonCyclic}. Hence $Z(Q)=1$.

By Theorem \ref{Th:Greer}, $Q$ is solvable. Let $S=Q'<Q$. We have $1<S$, else $Q$ is an abelian group. Let $|S|=s$ and $\{s,t\}=\{p,q\}$. Then $|Q/S|=t$, and both $S$ and $Q/S$ are cyclic groups of prime order. Let $x\in Q\setminus S$. Then $|\langle xS\rangle| = |Q/S|=t$, so $t$ divides $|x|$. By Theorem \ref{Th:ACauchyLagrange}, either $|x|=st=pq$ or $|x|=t$. If $|x|=pq$ then $Q=\langle x\rangle$ is a group, a contradiction. Hence $|x|=t$.

Let $(Q,\circ)$ be the associated left Bruck loop. By Proposition \ref{Pr:NR}, $(Q,\circ)$ contains a unique subloop of order $p$. Since powers in $(Q,\circ)$ and $(Q,\cdot)$ coincide, it follows that $Q$ contains precisely $p-1$ elements of order $p$. Hence $s=p$.
\end{proof}

We will need the following two results:

\begin{theorem}[\cite{KeNi}]\label{Th:KN}
Let $Q$ be a loop such that $\inn{Q}$ is a cyclic group. Then $Q$ is an abelian group.
\end{theorem}

\begin{theorem}[Albert]\label{Th:Albert}
Let $S$ be a normal subgroup of $Q$, and let $L_S = \setof{L_x}{x\in S}$. For a permutation group $G$ on $Q$, let $G_S = \setof{\varphi\in G}{\varphi|_S=\mathrm{id}_S}$ and $G_{Q/S} = \setof{\varphi\in G}{\varphi(xS) = xS\text{ for every }x\in Q}$. Then $\mlt{Q}_S = L_S\cdot \inn{Q}$, $\mlt{Q}_{Q/S} = L_S\cdot\inn{Q}_{Q/S}$ and $\inn{Q/S}\cong (\mlt{Q}_S)/(\mlt{Q}_{Q/S})$.
\end{theorem}

\begin{proposition}\label{Pr:Innpq}
Let $Q$ be a nonassociative commutative automorphic loop of order $pq$ with odd primes $p>q$. Then there is a normal subgroup $C\cong \mathbb Z_p$ of $\inn{Q}$ such that $\inn{Q}/C$ is a cyclic group of order dividing $p-1$.
\end{proposition}
\begin{proof}
Let $S$ be the unique normal subgroup of order $p$ in $Q$, whose existence is guaranteed by Lemma \ref{Lm:pq}. Consider the mapping $f:\inn{Q}\to\aut{S}$, $f(\varphi) = \varphi|_S$. Since $\varphi|_S\psi|_S(x) = \varphi|_S(\psi(x)) = \varphi(\psi(x)) = (\varphi\psi)|_S(x)$ for every $x\in S$, the mapping $f$ is a homomorphism. Its kernel is equal to $C=\setof{\varphi\in\inn{Q}}{\varphi|_S=\mathrm{id}_S}$. Now, $\aut{S}\cong\aut{\mathbb Z_p}\cong\mathbb Z_{p-1}$ is cyclic, so $\inn{Q}/C\le\aut{S}$ is a cyclic group of order dividing $p-1$. If $C$ is trivial, we deduce that $\inn{Q}$ is cyclic and Theorem \ref{Th:KN} then implies that $Q$ is an abelian group, a contradiction. Thus $C$ is nontrivial.

Let $S=\langle s\rangle$ and fix $t\in Q\setminus S$. Since $L_s(St) = s(St) = (sS)t = St$, the mapping $\psi = L_s|_{St}$ is a bijection on $St$. We claim that $\psi$ is a $p$-cycle. Suppose this is not the case. Since $\psi$ has no fixed points and $p$ is a prime, $\psi$ must contain nontrivial cycles of distinct lengths. Then a suitable power of $\psi$, say $\psi^i$, has more than $1$ but less than $p$ fixed points. Without loss of generality, let $t$ be a fixed point of $\psi^i$. Then $\alpha = L_t^{-1}L_s^i L_t\in\mlt{Q}$ fixes $1$. Thus $\alpha\in\inn{Q}\le\aut{Q}$, and $\alpha|_S\in\aut{S}$. Moreover, since $\alpha|_S$ is conjugate to $\psi^i$, they have the same cycle structure. The fixed points of $\alpha|_S$ then determine a nontrivial proper subgroup of $S\cong\mathbb Z_p$, a contradiction.

Since $Q/S$ is of prime order $q$, it is an abelian group and $\inn{Q/S}=1$. Then Theorem \ref{Th:Albert} gives $1 = \inn{Q/S}\cong (L_S\cdot\inn{Q})/(L_S\cdot\inn{Q}_{Q/S})$, so $\inn{Q} = \inn{Q}_{Q/S}$. In other words, every $\varphi\in\inn{Q}$ satisfies $\varphi(xS)=xS$ for every $x\in Q$.

Consider $1\ne\varphi\in C$. Then $\varphi$ is determined by the value on $t$, and $t\ne\varphi(t)\in St$. Because $\psi = L_s|_{St}$ is a $p$-cycle, there exists some $0<j<p$ such that $\psi^j(t) = \varphi(t)$. Furthermore, $\varphi(s^kt) = s^k\varphi(t) = L_{s^k}\psi^j(t) = \psi^jL_{s^k}(t) = \psi^j(s^kt)$ by Proposition \ref{Pr:PowerAssociative}, so $\varphi|_{St} = \psi^j$. Because $\psi^j$ is a $p$-cycle and $\varphi^k|_{St} = \psi^{jk}$ for every $k$, the elements $\varphi$, $\varphi^2$, $\dots$, $\varphi^p=1$ are distinct and account for all elements of $C$. Hence $C\cong\mathbb Z_p$.
\end{proof}

\begin{construction}[{{\cite[Propositions 3.1 and 3.6]{DrapalMetacyclic}}}]\label{Co:Drapal}
Let $p$ be an odd prime and $t\in\mathbb Z_p$. Define a partial map $f_t:\mathbb Z_p\to\mathbb Z_p$ by $f_t(x) = (x+1)(tx+1)^{-1}$. Suppose that for every $i\ge 1$ the value $f_t^i(0)$ is defined and there is a unique $x\in\mathbb Z_p$ such that $f_t^i(x)=0$. Let $d = |\setof{f_t^i(0)}{i\ge 1}|$. Then $\mathbb Z_p\times\mathbb Z_d$ with multiplication
\begin{displaymath}
    (i,a)(j,b) = (i+j,(a+b)(1+tf_t^i(0)f_t^j(0))^{-1})
\end{displaymath}
is a commutative automorphic loop.
\end{construction}

\begin{proposition}[\cite{JeSi}]
Construction \ref{Co:Drapal} yields a nonassociative commutative automorphic loop of order $pq$ for odd primes $p>q$ if and only if $q$ divides $p^2-1$, in which case it yields only one such loop up to isomorphism.
\end{proposition}

Thanks to Proposition \ref{Pr:Innpq}, all commutative automorphic loops of order $pq$ could be classified by the \emph{tour de force} of classifying all loops with trivial center and metacyclic inner mapping group, a program of Dr\'apal that is nearing completion (see, for instance, \cite{DrapalOrbits}). Another, perhaps easier approach, is to classify all left Bruck loops of order $pq$, and then use Theorem \ref{Th:GreerEquiv}. In particular, if there is a unique nonassociative left Bruck loop of order $pq$ and $q$ divides $p^2-1$, then it must correspond to a unique nonassociative commutative automorphic loop of order $pq$, constructed by Construction \ref{Co:Drapal}.

\subsection{Commutative automorphic loops of order $p^3$}

\begin{proposition}[\cite{JeKiVoConstructions}]\label{Pr:4p2}
Let $p$ be an odd prime and $Q$ a commutative automorphic loop. If $|Q|\in\{p$, $2p$, $4p$, $p^2$, $2p^2$, $4p^2\}$ then $Q$ is an abelian group.
\end{proposition}
\begin{proof}
By Theorem \ref{Th:Decomposition}, it suffices to prove that all commutative automorphic loops $Q$ of odd order $p$ and $p^2$ are groups. For $|Q|=p$ this is a special case of Corollary \ref{Cr:pAut}, for instance. When $|Q|=p^2$ then $Z(Q)$ is nontrivial by Theorem \ref{Th:Greer}, and the case $|Z(Q)|=p$ is excluded by Lemma \ref{Lm:NonCyclic}.
\end{proof}

In view of Proposition \ref{Pr:4p2}, commutative automorphic loops of order $p^3$ (for any prime $p$) are of interest. As above, we can easily show that if such a loop is nonassociative of odd order $p^3$ then $Z(Q)\cong\mathbb Z_p$ and $Q/Z(Q)\cong\mathbb Z_p\times\mathbb Z_p$. There are commutative automorphic loops of order $8$ with trivial center \cite{JeKiVoConstructions}.

Consider the following construction of \cite{JeKiVoConstructions}. Let $n\ge 2$ be an integer. The \emph{overflow indicator} $(.,.)_n:\mathbb Z_n\times\mathbb Z_n\to \{0,1\}$ is defined by
\begin{displaymath}
    (x,y)_n=\left\{\begin{array}{ll}1,&\text{ if $x+y\ge n$},\\0,&\text{otherwise}.\end{array}\right.
\end{displaymath}
For $a$, $b\in\mathbb Z_n$, define $\mathcal Q_{a,b}(\mathbb Z_n)$ on $\mathbb Z_n\times\mathbb Z_n\times \mathbb Z_n$ by
\begin{displaymath}
    (x_1,x_2,x_3)(y_1,y_2,y_3) = (x_1+y_1+(x_2+y_2)x_3y_3+a(x_2,y_2)_n+b(x_3,y_3)_n,\,x_2+y_2,\,x_3+y_3).
\end{displaymath}
Then $\mathcal Q_{a,b}(\mathbb Z_n)$ is a commutative automorphic loop of order $n^3$, $Z(Q)=N_\ell(Q)=\mathbb Z_n\times 0\times 0$, and $N_m(Q) = \mathbb Z_n\times\mathbb Z_n\times 0$.

It turns out that all nonassociative commutative automorphic loops of odd order $p^3$ are of the form $\mathcal Q_{a,b}(\mathbb Z_p)$. This was shown by De Barros, Grishkov and the author, who studied quotients of free $2$-generated nilpotent class $2$ commutative automorphic loops and also proved:

\begin{theorem}[\cite{BaGrVo}]
For every prime $p$, there are precisely $7$ commutative automorphic loops of order $p^3$ up to isomorphism, including the three abelian groups $\mathbb Z_{p^3}$, $\mathbb Z_{p^2}\times\mathbb Z_p$ and $\mathbb Z_p\times\mathbb Z_p\times\mathbb Z_p$.
\end{theorem}

The structure of the free $2$-generated commutative automorphic loop of nilpotency class $2$ can be found in \cite[Theorem 2.3]{BaGrVo}, which is proved by careful associator calculus. Lemma \ref{Lm:KeyAssoc} below gives some insight, and once again shows that the middle nucleus is of key importance in automorphic loops.

Recall that the \emph{associator} $(x,y,z)$ is defined by $(xy)z = x(yz)\cdot (x,y,z)$.

\begin{lemma}[{{\cite[Lemmas 2.1 and 2.2]{BaGrVo}}}]\label{Lm:KeyAssoc}
Let $Q$ be a commutative loop of nilpotency class $2$ (that is, $Q/Z(Q)$ is an abelian group). Then $(x,y,x)=1$, $(x,y,z)=(z,y,x)^{-1}$ and $(x,y,z)(y,z,x)(z,x,y)=1$ for every $x$, $y$, $z\in Q$. Moreover, $Q$ is automorphic if and only if $(xy,u,v) = (x,u,v)(y,u,v)$ for every $x$, $y$, $u$, $v\in Q$.

In the automorphic case, we have $(xy,u,v) = (x,u,v)(y,u,v)$, $(x,y,uv)=(x,y,u)(x,y,v)$, and $(x,yu,v) = (x,v,y)(x,v,u)(y,x,v)(u,x,v)$.
\end{lemma}

The structure of the free $2$-generated commutative automorphic loop of nilpotency class $3$ is also known, cf. \cite[Theorem 5.4]{BaGrVoClass3}.

\subsection{Two constructions of automorphic loops}

We conclude the lecture notes with two constructions of automorphic loops.

\begin{construction}[\cite{GrRaVo}]\label{Co:Module}
Let $R$ be a commutative ring, $V$ an $R$-module and $E=\mathrm{End}_R(V)$ the ring of $R$-endomorphisms of $V$. Let $(W,+)\le (E,+)$ be such that
\begin{enumerate}
\item[(i)] $ab=ba$ for every $a$, $b\in W$, and
\item[(ii)] $1+a$ is invertible for every $a\in W$.
\end{enumerate}
Define multiplication on $W\times V$ by
\begin{displaymath}
    (a,u)(b,v) = (a+b,\,(1+b)(u) + (1-a)(v)).
\end{displaymath}
Then $(W\times V,\cdot)$ is an automorphic loop.
\end{construction}

A special case of this construction was first given in \cite{JeKiVoNilp} in an effort to shed some light on automorphic loops of order $p^3$. (Automorphic loops of order $p^2$ are known to be groups by \cite{Csorgo} or by \cite[Theorem 6.1]{KiKuPhVo}.) A slight variation on Construction \ref{Co:Module} was also given in \cite{Nagy} in characteristic $2$.

An important special case of Construction \ref{Co:Module} can be given as follows: Let $R=k<K=V$, where $k<K$ is a field extension. Let $W$ be a $k$-subspace of $K$ such that $k1\cap W=0$. We can identify $a\in W$ with the $k$-endomorphism of $K$ given by $b\mapsto ba$ (the right translation by $a$ in $(K,\cdot)$). Then it is easy to see (cf. \cite{GrRaVo}) that the conditions (i) and (ii) of Construction \ref{Co:Module} are satisfied, and we obtain an automorphic loop $Q_{k<K}(W) = Q_{R,V}(W)$ on $W\times K$.

Let us come back to automorphic loops of order $p^3$. In order to obtain them as loops $Q_{k<K}(W)$, we choose $k=\mathbb F_p$ to be the field of order $p$ and $K=\mathbb F_{p^2}$ a quadratic field extension of $k$. If $p$ is odd, we can find all suitable $k$-subspaces $W$ as follows: The field $K$ can be identified with $\setof{x+y\sqrt{d}}{x,\,y\in k}$, where $d\in k$ is not a square. Let
\begin{displaymath}
    W_0 = k\sqrt d\text{ and } W_a= k(1+a\sqrt{d})\text{ for $0\ne a\in k$.}
\end{displaymath}
Then every $W_a$ is a $1$-dimensional $k$-subspace of $K$ such that $k1\cap W_a=0$. Conversely, if $W$ is a $1$-dimensional $k$-subspace of $K$ such that $k1\cap W=0$, there is $a+b\sqrt{d}$ in $W$ with $a$, $b\in k$, $b\ne 0$. If $a=0$ then $W=W_0$. Otherwise $a^{-1}(a+b\sqrt{d}) = 1+a^{-1}b\sqrt{d}\in W$, and $W=W_{a^{-1}b}$. Hence there is a one-to-one correspondence between the elements of $k$ and $1$-dimensional $k$-subspaces $W$ of $K$ satisfying $k1\cap W=0$, given by $a\mapsto W_a$.

\begin{proposition}[\cite{GrRaVo}]\label{Pr:p3}
Let $p$ be a prime and $\mathbb F_p=k<K=\mathbb F_{p^2}$.
\begin{enumerate}
\item[(i)] Suppose that $p$ is odd. If $a$, $b\in k$, then the automorphic loops $Q_{k<K}(W_a)$, $Q_{k<K}(W_b)$ of order $p^3$ are isomorphic if and only if $a=\pm b$. In particular, there are $(p+1)/2$ pairwise nonisomorphic automorphic loops of order $p^3$ of the form $Q_{k<K}(W)$, where we can take $W\in \{W_a:0\le a\le (p-1)/2\}$.
\item[(ii)] Suppose that $p=2$. Then there are $2$ pairwise nonisomorphic automorphic loops of order $p^3$ of the form $Q_{k<K}(W)$.
\end{enumerate}
\end{proposition}

We do not claim that Proposition \ref{Pr:p3} accounts for all automorphic loops of order $p^3$.

\bigskip

Finally, we present a construction reminiscent of generalized dihedral groups.

\begin{construction}[\cite{Aboras}]\label{Co:Dihedral}
Let $(G,+)$ be an abelian group and $m>1$ an even integer. Let $\alpha\in\aut{G}$. Define multiplication on $\mathbb Z_m\times G$ by\begin{displaymath}
    (i,u)(j,v) = (i+j, \alpha^{ij}((-1)^ju+v)).
\end{displaymath}
Then the resulting loop $\mathrm{Dih}(m,G,\alpha)$ is automorphic if and only if $m=2$ or $\alpha^2=1$.
\end{construction}

Aboras \cite{AborasThesis} obtained many structural properties of the dihedral-like automorphic loops $\mathrm{Dih}(m,G,\alpha)$, which are of interest because they account for many small automorphic loops.

The special case of Construction \ref{Co:Dihedral} with $m=2$ was originally introduced in \cite{KiKuPhVo}, and the following result was obtained there:

\begin{theorem}[{{\cite[Corollary 9.9]{KiKuPhVo}}}]
Let $p$ be an odd prime, and let $Q$ be a loop of order $2p$. Then $Q$ is automorphic if and only if it is isomorphic to the cyclic group $\mathbb Z_{2p}$ or to a dihedral-like loop $\mathrm{Dih}(2,\mathbb Z_p,\alpha)$ for some $\alpha\in\aut{\mathbb Z_p}$. There are precisely $p$ pairwise nonisomorphic automorphic loops of order $2p$.
\end{theorem}

Coming back full circle, the automorphic loop $Q_6$ from the introduction is isomorphic to the loop $\mathrm{Dih}(2,\mathbb Z_3,\alpha)$, where $\alpha$ is the unique nontrivial automorphism of $\mathbb Z_3$.

\section*{Open problems}
\setcounter{section}{4}
\setcounter{theorem}{0}
\setcounter{subsection}{0}

\begin{problem}
Is there a finite simple nonassociative automorphic loop?
\end{problem}

\begin{problem}
Is there an automorphic loop of odd order with trivial middle nucleus?
\end{problem}

\begin{problem}
If $Q$ is a finite automorphic loop and $H\le Q$, does $|H|$ divide $|Q|$?
\end{problem}

Let $p$ be a prime.

\begin{problem}
Find an elementary proof of the fact that automorphic loops of order $p^2$ are groups.
\end{problem}

\begin{problem}
Classify automorphic loops of order $p^3$.
\end{problem}

\begin{problem}
Classify commutative automorphic loops of order $p^4$.
\end{problem}

\begin{problem}
Classify left Bruck loops of order $pq$ and $p^2q$, where $p$, $q$ are distinct odd primes.
\end{problem}

\begin{problem}
Classify (commutative) automorphic loops of order $pq$ and $p^2q$, where $p$, $q$ are distinct odd primes.
\end{problem}

\begin{problem}
Study free commutative automorphic loops with $k$ free generators and of nilpotency class $n$. Already the cases $(k,n)=(2,4)$ and $k\ge 3$ are open.
\end{problem}

\begin{problem}
Study in detail the mapping $\Phi:(Q,\cdot)\mapsto (Q,\circ)$ that associates a uniquely $2$-divisible left Bruck loop $(Q,\circ)$ to a uniquely $2$-divisible automorphic loop $(Q,\cdot)$ via $x\circ y = (x^{-1}\ldiv y^2x)^{1/2}$. In particular, what is the image of $\Phi$? If $(Q,\circ)\in\mathrm{im}(\Phi)$, is there also a commutative automorphic loop $(Q,\cdot)$ such that $(Q,\circ) = \Phi(Q,\cdot)$?
\end{problem}

\begin{problem}
Can Proposition \ref{Pr:Greer} be extended from left Bruck loops of odd order to uniquely $2$-divisible left Bruck loops, perhaps under different correspondence?
\end{problem}

\begin{problem}
Let $(Q,+,[.,.])$ be an algebra in which the condition \eqref{Eq:Wright1} holds, and let $(Q,\cdot)$ be the associated linear loop with multiplication $x\cdot y = x+y-[x,y]$. Characterize when $(Q,\cdot)$ is an automorphic loop (beyond the obvious equational characterization). Are there interesting classes of algebras for which $(Q,\cdot)$ is always automorphic?
\end{problem}

\begin{problem}
Let $(Q,+,[.,.])$ be a Lie ring satisfying \eqref{Eq:Wright1}. Characterize when the associated linear loop $(Q,\cdot)$ is automorphic (beyond the obvious equational characterization).
\end{problem}

An alternative theory of solvability in loop theory has been developed in \cite{StVo}, based on concepts from universal algebra (congruence modular varieties). Let us call this solvability \emph{congruence solvability}. Congruence solvability is in general a stronger concept than solvability. To see whether congruence solvability is the right concept for loops, theorems previously proved for (classical) solvability in loops should be revisited. In particular:

\begin{problem}
Are left Bruck (Moufang, commutative automorphic, automorphic) loops of odd order congruence solvable?
\end{problem}

\section*{Acknowledgment}

I thank Michael Kinyon for useful conversations on the Greer correspondence, P\v{r}emysl Jedli\v{c}ka for comments on Construction \ref{Co:Drapal}, Ale\v{s} Dr\'apal for the key idea in the proof of Proposition \ref{Pr:Innpq}, and an anonymous referee for a few improvements to the presentation.

\end{document}